\documentclass[12pt]{amsart}
\usepackage[margin=1.2in]{geometry}
\usepackage{graphicx,latexsym}
\usepackage{tikz}
\usepackage{amsfonts, amssymb, amsmath, amsthm, bm, hyperref,verbatim}
\usepackage{tikz}
\usetikzlibrary{matrix,arrows}
\usepackage{mathrsfs}
\allowdisplaybreaks

\newcommand{\N}{\mathbb{N}}
\newcommand{\Z}{\mathbb{Z}}
\newcommand{\R}{\mathbb{R}}
\newcommand{\C}{\mathbb{C}}

\newcommand{\M}{\mathfrak{M}}
\newcommand{\W}{\mathscr{W}}

\newcommand{\csn}{\operatorname{csn}}

\newcommand{\beq}{\begin{eqnarray}}
\newcommand{\eeq}{\end{eqnarray}}

\newcommand{\beqs}{\begin{eqnarray*}}
\newcommand{\eeqs}{\end{eqnarray*}}

\newtheorem{theorem}{Theorem}[section]
\newtheorem{proposition}[theorem]{Proposition}
\newtheorem{lemma}[theorem]{Lemma}
\newtheorem{corollary}[theorem]{Corollary}

\theoremstyle{definition}

\theoremstyle{remark}
\newtheorem{remark}[theorem]{Remark}

\numberwithin{equation}{section}

\usetikzlibrary{arrows,chains,matrix,positioning,scopes}
\makeatletter
\tikzset{join/.code=\tikzset{after node path={%
\ifx\tikzchainprevious\pgfutil@empty\else(\tikzchainprevious)%
edge[every join]#1(\tikzchaincurrent)\fi}}}
\makeatother
\tikzset{>=stealth',every on chain/.append style={join},
         every join/.style={->}}
\tikzstyle{labeled}=[execute at begin node=$\scriptstyle,
   execute at end node=$]

\begin{document}
\title[The nuclearity of Gelfand-Shilov spaces]{The nuclearity of Gelfand-Shilov spaces and kernel theorems}

\author[A. Debrouwere]{Andreas Debrouwere}
\thanks{A. Debrouwere was supported by  FWO-Vlaanderen through the postdoctoral grant 12T0519N}

\author[L. Neyt]{Lenny Neyt}
\thanks{L. Neyt gratefully acknowledges support by Ghent University through the BOF-grant 01J11615.}

\author[J. Vindas]{Jasson Vindas}
\thanks {J. Vindas was supported by Ghent University through the BOF-grants 01J11615 and 01J04017.}

\address{Department of Mathematics: Analysis, Logic and Discrete Mathematics\\ Ghent University\\ Krijgslaan 281\\ 9000 Gent\\ Belgium}
\email{andreas.debrouwere@UGent.be}
\email{lenny.neyt@UGent.be}
\email{jasson.vindas@UGent.be}


\begin{abstract}
We study the nuclearity of the Gelfand-Shilov spaces $\mathcal{S}^{(\mathfrak{M})}_{(\mathscr{W})}$  and  $\mathcal{S}^{\{\mathfrak{M}\}}_{\{\mathscr{W}\}}$, defined via  a weight (multi-)sequence system $\mathfrak{M}$ and a weight function system $\mathscr{W}$.  We obtain characterizations of nuclearity for these function spaces that are counterparts of those for K\"{o}the sequence spaces. As an application, we  prove new kernel theorems.
 Our general framework allows for a unified treatment of the Gelfand-Shilov spaces $\mathcal{S}^{(M)}_{(A)}$  and  $\mathcal{S}^{\{M\}}_{\{A\}}$ (defined via  weight sequences $M$ and $A$) and the Beurling-Bj\"orck spaces $\mathcal{S}^{(\omega)}_{(\eta)}$  and  $\mathcal{S}^{\{\omega\}}_{\{\eta\}}$ (defined via weight functions $\omega$ and $\eta$). Our results cover anisotropic cases as well.
\end{abstract}
\maketitle
\section{Introduction}
Nuclear spaces play a major role in functional analysis. One of their key features is the validity of abstract Schwartz kernel theorems, which often allows for the representation and study of important classes of continuous linear mappings via kernels. Therefore, establishing whether a given function space is nuclear becomes a central question from the point of view of both applications and understanding its locally convex structure.

In the case of weighted Fr\'{e}chet spaces of smooth functions on $\mathbb{R}^{d}$, the nuclearity question has been completely settled. Let $\W = (w_n)_{n \in \N}$ be a sequence of positive continuous functions on $\R^d$ such that $1 \leq w_1 \leq w_2 \leq \cdots$. Consider the associated Gelfand-Shilov spaces of smooth functions \cite{Gelfand-Shilov2}
$$
\mathcal{K}^q(\W) = \{ \varphi \in C^\infty(\R^d) \, | \,  \max_{|\alpha| \leq n} \| \varphi^{(\alpha)} w_n \|_{L^q} < \infty \quad \forall n \in \N \}, \qquad q \in [1,\infty],
$$
endowed with their natural Fr\'echet space topologies. If $\W$ satisfies the mild regularity hypothesis
\begin{equation*}
\tag{$\operatorname{wM}$}
 \forall n \in \N \, \exists m \in \N \, \exists C > 0 \, \forall x \in \R^d \, : \,  \sup_{|y| \leq 1} w_n(x + y) \leq C w_m(x),
\end{equation*}
then $\mathcal{K}^q(\W)$ is nuclear if and only if $\W$ satisfies the condition
\[
\tag{$\operatorname{N}$} \forall n \in \N \, \exists m \in \N \,:\, w_n / w_m \in L^{1}.
\]
In fact, this result follows from Vogt's sequence space representation of $\mathcal{K}^q(\W)$ \cite[Theorem 3.1]{Vogt} and the well-known corresponding characterization of nuclearity for K\"othe sequence spaces \cite[Proposition 28.16]{M-V}. Condition $(\operatorname{N})$ appears already in the work of Gelfand and Shilov, who proved the nuclearity of $\mathcal{K}^\infty(\W)$ under it and some extra  regularity assumptions in a direct fashion \cite[p.~181]{Gelfand-Shilov3}.

The goal of this article is to establish the analogue of the above result for Gelfand-Shilov spaces of ultradifferentiable functions, both of Beurling and Roumieu type. These 
spaces, also known as spaces of type $\mathcal{S}$, were introduced in the context of parabolic initial-value problems \cite{Gelfand-Shilov3, Gelfand-Shilov2}, and thereafter turned out to be the right framework for the analysis of decay and regularity properties of global solutions to large classes of linear and semi-linear PDE on $\mathbb{R}^{d}$.  We refer to the monograph \cite{N-R} and the survey article \cite{Gramchev2016} for accounts on applications of Gelfand-Shilov spaces; see also \cite{C-T,Prangoski2013} for global pseudo-differential calculus in this setting. The study of nuclearity for spaces of type $\mathcal{S}$ goes back to Mityagin \cite{Mitjagin1960}, and has recently captured much attention  \cite{B-J-O2019,B-J-O,B-J-O-S, D-N-V-NBB, P-P-V}; particularly, in connection with applications to microlocal analysis of pseudo-differential operators and the convolution theory for generalized functions. We mention that in some cases nuclearity becomes a straightforward consequence of sequence space representations provided  by eigenfunction expansions with respect to various PDO \cite{C-G-P-R,Langenbruch2006,V-V2016}. 
However, such representations are not available for all Gelfand-Shilov spaces.  We deal in this article with the latter situation.

We shall work here with the notion of ultradifferentiability defined through weight matrices \cite{R-S-Composition}, called weight sequence systems in the present article.
In particular, as explained in \cite{R-S-Composition}, this leads to a unified treatment of classes of ultradifferentiable functions defined via weight sequences \cite{Komatsu} and via weight functions \cite{B-M-T}. Moreover, we further extend the considerations from \cite{R-S-Composition} to  multi-indexed weight sequence systems in order to cover the anisotropic case as well. Our spaces are defined as follows. Given a family $\M = \{ M^{\lambda} \, | \, \lambda \in \R_{+} \}$ of (multi-indexed) sequences $M^\lambda = (M^\lambda_\alpha)_{\alpha \in \N^{d}}$ of positive numbers such that  $M^{\lambda} \leq M^{\mu}$ for all $\lambda \leq \mu$ and a family $\mathscr{W} = \{ w^{\lambda} \,|\, \lambda \in \R_{+} \}$ of positive continuous functions $w^{\lambda}$ on $\R^{d}$ such that $1 \leq w^{\lambda} \leq w^{\mu}$ for all  $\mu \leq \lambda$, we consider the  Gelfand-Shilov spaces 
$$
\mathcal{S}^{(\mathfrak{M})}_{(\mathscr{W}),q} = \varprojlim_{\lambda \to 0^+} \mathcal{S}^{M^\lambda}_{w^\lambda,q}, \qquad \mathcal{S}^{\{\mathfrak{M}\}}_{\{\mathscr{W}\},q} = \varinjlim_{\lambda \to \infty} \mathcal{S}^{M^\lambda}_{w^\lambda,q}, \qquad q \in [1,\infty],
$$
where $\mathcal{S}^{M^\lambda}_{w^\lambda,q}$ is the Banach space consisting of all $\varphi \in C^\infty(\R^d)$ such that
$$
\| \varphi \|_{\mathcal{S}^{M^\lambda}_{w^\lambda,q}} = \sup_{\alpha \in \N^d} \frac{\| \varphi^{(\alpha)} w^\lambda \|_{L^q}}{M^\lambda_{\alpha}}  < \infty.
$$
We use $\mathcal{S}^{[\mathfrak{M}]}_{[\mathscr{W}],q}$ as a common notation for  $\mathcal{S}^{(\mathfrak{M})}_{(\mathscr{W}),q}$  and  $\mathcal{S}^{\{\mathfrak{M}\}}_{\{\mathscr{W}\},q}$. In Section \ref{sect-4}, we give  sufficient conditions for $\mathcal{S}^{[\mathfrak{M}]}_{[\mathscr{W}],q}$ to be nuclear in terms of $\M$ and $\W$; see Theorem \ref{ultra-suff}. Actually, we show that for an important class of weight sequence systems our hypotheses also become necessary, providing a full characterization of nuclearity in such a case; see Theorem \ref{ultra-char}. Moreover, nuclearity is related to the identity $\mathcal{S}^{[\mathfrak{M}]}_{[\mathscr{W}],q} = \mathcal{S}^{[\mathfrak{M}]}_{[\mathscr{W}],r}$, $q \neq r$. Section \ref{sect-3} is devoted to studying when this equality holds; see Theorem \ref{ft-char-N} and Theorem \ref{ultra-char-2}.  A useful feature of our approach is that our considerations are stable under tensor products. We shall exploit this fact to derive new kernel theorems for Gelfand-Shilov spaces in Section \ref{section kernel}. Note that these kernel theorems are `global' counterparts of Petzsche's results from \cite{Petzsche}. 

We end this introduction by stating two samples of our results in the particular but important case of the Gelfand-Shilov spaces $\mathcal{S}^{[M]}_{[A],q}$ arising from single weight sequences $M=(M_{p})_{p\in\mathbb{N}}$ and $A=(A_{p})_{p\in\mathbb{N}}$ . We refer to Section \ref{sect-2} and Section \ref{sect-3} for unexplained notions.

\begin{theorem}\label{intro-1}
Let $M=(M_{p})_{p\in\mathbb{N}}$ and $A=(A_{p})_{p\in\mathbb{N}}$ be two weight sequences satisfying $(M.1)$. Suppose that $\mathcal{S}^{[M]}_{[A],q} \neq \{0\}$ for some $q \in [1,\infty]$. Then, the following statements are equivalent:
\begin{itemize}
\item[$(i)$] $M$ and $A$ both satisfy $(M.2)'$.
\item[$(ii)$] $\mathcal{S}^{[M]}_{[A],q}$ is nuclear for all $q \in [1,\infty]$.
\item[$(iii)$] $\mathcal{S}^{[M]}_{[A],q}$ is nuclear for some $q \in [1,\infty]$.
\end{itemize}
\end{theorem}

The implication $(i)\Rightarrow (ii)$ in the Roumieu case is actually contained in Mityagin's work \cite{Mitjagin1960}. The same topic is treated in \cite[Proposition 2.11 and Remark 2.13]{P-P-V} including also the Beurling case but making use of stronger assumptions on the weight sequences. On the other hand,  to the best of our knowledge, the converse, i.e., the implication $(iii)\Rightarrow (i)$ appears to be new even in the particular case we have stated in Theorem \ref{intro-1}. Notice that $(i)$ is well-known  to yield  the equality $\mathcal{S}^{[M]}_{[A],q}$ =  $\mathcal{S}^{[M]}_{[A],r}$ as locally convex spaces for all $q,r \in  [1,\infty]$, see e.g.\ \cite[Theorem 2.5.2]{C-K-P}. Interestingly, when we fix the weight sequence $M$ and assume that it satisfies $(M.1)$ and $(M.2)'$ (cf. Subsection \ref{subsection weight sequence systems}), then we can also get back from this statement to nuclearity. 

\begin{theorem}\label{intro-2}
Let $M=(M_{p})_{p\in\mathbb{N}}$ be a weight sequence satisfying $(M.1)$ and $(M.2)'$. Let $A  = (A_{p})_{p\in\mathbb{N}} $ be a weight sequence satisfying $(M.1)$. Suppose that  $\mathcal{S}^{[M]}_{[A],q} \neq \{0\}$ for some $q \in [1,\infty]$. Then, the following statements are equivalent:
\begin{itemize}
\item[$(i)$] $A$ satisfies $(M.2)'$.
\item[$(ii)$] $\mathcal{S}^{[M]}_{[A],q}$ is nuclear for all $q \in [1,\infty]$.
\item[$(iii)$] $\mathcal{S}^{[M]}_{[A],q}$ is nuclear for some $q \in [1,\infty]$.
\item[$(iv)$] $\mathcal{S}^{[M]}_{[A],q}$ =  $\mathcal{S}^{[M]}_{[A],r}$ as locally convex spaces for all $q,r \in  [1,\infty]$.
\item[$(v)$] $\mathcal{S}^{[M]}_{[A],q}$ =  $\mathcal{S}^{[M]}_{[A],r}$ as sets for some  $q,r \in  [1,\infty]$ with $q \neq r$.
\end{itemize}
\end{theorem} 

It should be noticed that Theorem \ref{intro-2}  corresponds  to the classical characterization of nuclearity for K\"othe sequence spaces  (cf. \cite[Proposition\ 28.16]{M-V}) in the setting of Gelfand-Shilov spaces. As a corollary of our more general considerations, we shall also obtain in Section \ref{sect-4} an analogue of Theorem \ref{intro-2} for spaces of Beurling-Bj\"orck type \cite{bjorck66} with ultradifferentiable component defined through a Braun-Meise-Taylor weight function \cite{B-M-T}; in fact, Theorem \ref{intro-3} considerably improves recent results by Boiti et al. \cite{B-J-O2019,B-J-O, B-J-O-S}. In this regard, we mention our recent article \cite{D-N-V-NBB}, whose results apply to Beurling-Bj\"orck spaces not necessarily arising from Braun-Meise-Taylor weight functions.

\section{Nuclear K\"othe sequence spaces}\label{sect-1}
K\"othe sequence spaces, both of echelon and co-echelon type \cite{M-V,B-S}, are essential tools in this work. We review in this short section when they are nuclear. 

Given a  sequence $a= (a_j)_{j \in \Z^d}$ of positive numbers, we define $l^q(\Z^d, a) = l^q(a) $, $q \in [1,\infty]$, as the weighted Banach sequence space consisting of all $c = (c_j)_{j \in \Z^d} \in \C^{\Z^d}$ such that
$$
\| c \|_{l^q(a)} = \left(\sum_{j \in  \Z^d} (|c_j|a_j)^q\right)^{1/q}  < \infty, \qquad q \in [1,\infty),
$$
and
$$
\| c \|_{l^\infty(a)} = \sup_{j \in  \Z^d} |c_j|a_j < \infty.
$$
Set $\R_+ = (0,\infty)$. A  \emph{K\"othe set}  is a family $A = \{ a^\lambda \, | \, \lambda \in \R_+ \}$ of sequences $a^\lambda$ of positive numbers such that $a^{\lambda}_{j} \leq a^\mu_j$ for all $j \in \Z^d$ and $\mu \leq \lambda$. We define the associated  \emph{K\"othe sequence spaces} as
$$
\lambda^q(A) = \varprojlim_{\lambda \to 0^+} l^q(a^\lambda), \qquad  \lambda^q\{A\} = \varinjlim_{\lambda \to \infty} l^q(a^\lambda), \qquad q \in [1,\infty].
$$
Note that $\lambda^q(A)$ is a Fr\'echet space, while $\lambda^q\{A\}$ is a regular $(LB)$-space, as follows from \cite[p.~ 80, Corollary 7]{Bierstedt}. We will use $\lambda^q[A]$ as a common notation for $\lambda^q(A)$ and $\lambda^q\{A\}$. In addition, we shall often first state assertions for $\lambda^q(A)$ followed in parenthesis by the corresponding statements for $\lambda^q\{A\}$. Similar conventions will be used for other spaces and notations. 

The nuclearity of the spaces $\lambda^q[A]$ can be characterized in terms of the following conditions on the K\"othe set $A$:
\begin{itemize}
\item[$(\operatorname{N})$] $\displaystyle \forall \lambda \in \R_+ \, \exists \mu \in \R_+ \, : \,  a^\lambda/a^\mu \in l^1$;
\item[$\{\operatorname{N}\}$] $\displaystyle \forall \mu \in \R_+ \, \exists \lambda \in \R_+ \, : \,   a^\lambda/a^\mu \in l^1$.
\end{itemize}
\begin{proposition}[{cf.\ \cite[Proposition 28.16]{M-V} and \cite[Proposition 15, p.~75]{Bierstedt})}] \label{nuclear-echelon}
Let  $A$ be a K\"othe set. The following statements are equivalent:
\begin{itemize}
\item[$(i)$] $A$ satisfies $[\operatorname{N}]$.
\item[$(ii)$] $\lambda^q[A]$ is nuclear for all $q \in  [1,\infty]$.
\item[$(iii)$] $\lambda^q[A]$ is nuclear for some $q \in  [1,\infty]$.
\item[$(iv)$] $\lambda^q[A] = \lambda^r[A]$ as locally convex spaces for all $q,r \in  [1,\infty]$.
\item[$(v)$] $\lambda^q[A] = \lambda^r[A]$ as sets for some $q,r \in  [1,\infty]$ with $q \neq r$.
\end{itemize}
\end{proposition}
\section{Weight sequence systems and weight function systems}\label{sect-2}
In this section, we define weight sequence systems (introduced in  \cite{R-S-Composition} under the name weight matrix) and weight function systems.
\subsection{Weight function systems} 
A positive continuous function $w$ on $\R^d$ is called a \emph{weight function} on $\R^d$ if  $w(x) \geq 1$ for all $x \in \R^d$. A \emph{weight function system} on $\R^d$ is a family $\mathscr{W} = \{ w^{\lambda} \,|\, \lambda \in \R_{+} \}$ of weight functions  $w^{\lambda}$ on $\R^d$ such that $w^{\lambda}(x) \leq w^{\mu}(x)$ for all $x \in \R^d$ and  $\mu \leq \lambda$. We consider the following conditions on a weight function system $\mathscr{W}$:
	\begin{itemize}
		\item[$(\operatorname{wM})$] $\forall \lambda \in \R_{+} ~ \exists \mu \in \R_{+} ~ \exists C > 0 ~ \forall x \in \R^{d} \,:\,  \sup_{|y| \leq 1} w^{\lambda}(x + y) \leq C w^{\mu}(x)$;
		\item[$\{\operatorname{wM}\}$] $\forall \mu \in \R_{+} ~ \exists \lambda \in \R_{+} ~ \exists C > 0 ~ \forall x \in \R^{d} \,:\, \sup_{|y| \leq 1} w^{\lambda}(x + y) \leq C w^{\mu}(x)$;
		\item[$(\operatorname{M})$] $\forall \lambda \in \R_{+} ~ \exists \mu, \nu \in \R_{+} ~ \exists C > 0 ~ \forall x, y \in \R^{d} \,:\, w^{\lambda}(x + y) \leq C w^{\mu}(x) w^{\nu}(y)$; 
		\item[$\{\operatorname{M}\}$] $\forall  \mu, \nu \in \R_{+} ~ \exists \lambda \in \R_{+} ~ \exists C > 0 ~ \forall x, y \in \R^{d} \,:\, w^{\lambda}(x + y) \leq C w^{\mu}(x) w^{\nu}(y)$; 
		\item[$(\operatorname{N})$] $\forall \lambda \in \R_{+} ~ \exists \mu \in \R_{+} \,:\, w^{\lambda} / w^{\mu} \in L^{1}$;
		\item[$\{\operatorname{N}\}$] $\forall \mu \in \R_{+} ~ \exists \lambda \in \R_{+} \,:\, w^{\lambda} / w^{\mu} \in L^{1}$.
	\end{itemize}
Clearly, $[\operatorname{M}]$ implies $[\operatorname{wM}]$.  Note that $[\operatorname{wM}]$ yields that
\begin{gather}
\label{joo}
\forall \lambda \in \R_{+} ~ \exists \lambda' \in \R_+ ~ \forall \mu' \in \R_+ ~ \exists \mu \in \R_{+} ~ (\forall \lambda' \in \R_{+} ~ \exists \lambda \in \R_+ ~ \forall \mu \in \R_+ ~ \exists \mu' \in \R_{+}) \\ \nonumber
 \exists C > 0 ~ \forall x \in \R^{d} \, : \,\sup_{|y| \leq 1}\frac{w^{\lambda}(x+y)}{w^{\mu}(x+y)} \leq C \frac{w^{\lambda'}(x)}{w^{\mu'}(x)}.
\end{gather}

We define the tensor product of a finite number of weight function systems  $\mathscr{W}_j = \{ w^{\lambda}_j \,|\, \lambda \in \R_{+} \}$ on $\R^{d_j}$, $j = 1, \ldots, k$,
as
$$
\mathscr{W}_1 \otimes \cdots \otimes \mathscr{W}_k = \{ w^{\lambda}_1\otimes \cdots \otimes w^{\lambda}_k \, | \, \lambda \in \R_+\},
$$
where $(w^{\lambda}_1 \otimes \cdots \otimes w^{\lambda}_k)(x) = w^{\lambda}_1(x_1) \cdots w^{\lambda}_k(x_k)$ for $x = (x_1, \ldots, x_k) \in \R^{d_1 + \ldots +d_k}$.
Note that $\mathscr{W}_1 \otimes \cdots \otimes \mathscr{W}_k$ satisfies $[\operatorname{wM}]$ ($[\operatorname{M}]$ and $[\operatorname{N}]$, respectively) if and only if each $\W_j$ does so.

The following lemma will be needed later on. We denote by  $C_{0}$ the space of continuous functions vanishing at $\infty$. 
	\begin{lemma}
		\label{lemma-C0}
		Let $\W$ be a weight function system satisfying $[\operatorname{wM}]$ and $[\operatorname{N}]$. Then,
			\[ \forall \lambda \in \R_{+} ~ \exists \mu \in \R_{+} ~ ( \forall \mu \in \R_{+} ~ \exists \lambda \in \R_{+} ) \, : \, w^{\lambda} / w^{\mu} \in L^{1} \cap C_{0}. \]
	\end{lemma}
	\begin{proof}
	This is a consequence of \eqref{joo}.
	\end{proof}
	
Given a weight function system $\W$, we associate to it the K\"{o}the set 
$$
A_{\W} = \{ (w^{\lambda}(j))_{j \in \Z^{d}} \, | \, \lambda \in \R_{+} \}.
$$ 
The next result shows that the notion $[\operatorname{N}]$ is unambiguous. 
	\begin{lemma}
		\label{l:KotheNiffWN}
		Let $\W$ be a weight function system satisfying $[\operatorname{wM}]$. Then, $\W$ satisfies $[\operatorname{N}]$ if and only if $A_{\W}$ satisfies $[\operatorname{N}]$.
	\end{lemma}	
	\begin{proof}
 	This again follows from \eqref{joo}.
	\end{proof}

\subsection{Weight sequence systems}\label{subsection weight sequence systems} A  sequence $M = (M_{\alpha})_{\alpha \in \N^{d}}$ of positive numbers is called a \emph{weight sequence}  on $\N^d$ if $\lim_{\alpha \to \infty} (M_{\alpha})^{1 / |\alpha |} =\infty$. We write $e_{j}$ for the standard coordinate unit vectors in $\mathbb{R}^{d}$, $j=1,\dots,d$. We consider the following conditions on a weight sequence $M$:
	\begin{itemize}
		\item[$(M.1)$] \emph{(log-convexity)} $M^2_{\alpha+e_{j}} \leq M_{\alpha} M_{\alpha+2e_{j}}$ for all $\alpha\in \mathbb{N}^{d}$ and $j \in \{1, \ldots, d\}$;
		\item[$(M.2)'$] \emph{(derivation-closedness)} $M_{\alpha + e_{j}} \leq C H^{|\alpha|} M_{\alpha}$ for all $\alpha\in \mathbb{N}^{d}$ and $j \in \{1, \ldots, d\}$, and some $C,H > 0$.
	\end{itemize}
The conditions $(M.1)$ and $(M.2)'$ are denoted by $(M_{\operatorname{lc}})$ and $(M_{\operatorname{dc}})$, respectively, in \cite{R-S-Composition}. We use here the standard notation from \cite{Komatsu}.  
The \emph{associated function} of $M$ is defined as
	\[ \omega_M(x) =  \sup_{\alpha \in \N^{d}}  \log \frac{|x^{\alpha}| M_{0}}{M_{\alpha}} , \qquad x\in\mathbb{R}^{d}. \]
Observe that 	$\omega_M(x) =  \omega_M(|x_1|, \ldots, |x_d|)$ for all $x = (x_1, \ldots, x_d) \in \R^d$. 

We define the tensor product of a finite number of  weight sequences $M_{j} = (M_{j,\alpha})_{\alpha \in \N^{d_j}}$ on $\N^{d_j}$, $j = 1, \ldots , k$, as 
$
M_1\otimes \dots \otimes M_k  = (M_{1,\alpha_{1}} \cdots M_{k,\alpha_{k}})_{(\alpha_1,\ldots, \alpha_k)\in \mathbb{N}^{d_1+\dots+d_k}}
$.
Note that $M_1\otimes \dots \otimes M_k$ satisfies $(M.1)$ ($(M.2)'$, respectively) if and only if this property holds for  each $M_j$. Moreover, 
$$
\omega_{M_{1}\otimes \dots\otimes M_{k}}(x)= \sum_{j=1}^{k}\omega_{M_j}(x_j), \qquad x=(x_1,\dots,x_k)\in \mathbb{R}^{d_1+\dots+d_{k}}.
$$

Next,  a \emph{weight sequence system} on $\R^d$ is a family  $\M = \{ M^{\lambda} \, | \, \lambda \in \R_{+} \}$ of weight sequences $M^\lambda$ on $\N^d$ satisfying $(M.1)$  such that $M_{\alpha}^{\lambda} \leq M_{\alpha}^{\mu}$ for all $\alpha \in \N^{d}$ and  $\lambda \leq \mu$. 
 We will work with some of the following conditions on a weight sequence system $\M$:
	\begin{itemize}
		\item[$(\operatorname{L})$] $\forall R > 0 ~ \forall \lambda \in \R_{+} ~ \exists \mu \in \R_{+} ~ \exists C > 0 ~ \forall \alpha \in \N^{d} \, : \, R^{|\alpha|} M_{\alpha}^{\mu} \leq C M_{\alpha}^{\lambda}$;
		\item[$\{\operatorname{L}\}$] $\forall R > 0 ~ \forall \mu \in \R_{+} ~ \exists \lambda \in \R_{+} ~ \exists C > 0 ~ \forall \alpha \in \N^{d} \, : \,  R^{|\alpha|} M_{\alpha}^{\mu} \leq C M_{\alpha}^{\lambda}$;
		\item[$(\mathfrak{M}.2)'$] $\forall \lambda \in \R_{+} ~ \exists \mu \in \R_{+} ~ \exists C,H >0 ~ \forall \alpha \in \N^{d}~\forall j \in \{ 1, \ldots, d \} \, : \, M_{\alpha + e_{j}}^{\mu} \leq CH^{|\alpha|} M_{\alpha}^{\lambda}$;
		\item[$\{\mathfrak{M}.2\}'$] $\forall \mu \in \R_{+} ~ \exists \lambda \in \R_{+} ~ \exists C,H >0 ~ \forall \alpha \in \N^{d}~\forall j \in \{ 1, \ldots, d \} \, : \, M_{\alpha + e_{j}}^{\mu} \leq CH^{|\alpha|} M_{\alpha}^{\lambda}$.
	\end{itemize}
The conditions $[\operatorname{L}]$ and $[\mathfrak{M}.2]'$ are denoted by $(\mathfrak{M}_{[\operatorname{L}]})$ and $(\mathfrak{M}_{[\operatorname{dc}]})$, respectively, in \cite{R-S-Composition}. Furthermore, $\M$ is called  \emph{accelerating} if $M_{\alpha+e_j}^{\lambda}/M_{\alpha}^{\lambda} \leq M_{\alpha+e_j}^{\mu}/M_{\alpha}^{\mu}$ for all $\alpha\in \N^d$, $j \in \{ 1, \ldots, d \}$, and  $\lambda \leq \mu$. 

We define the tensor product of a finite number of  weight sequence systems $\mathfrak{M}_{j}=\{M^{\lambda}_j  \;| \,\lambda\in \mathbb{R}_{+}\}$ on $\N^{d_j}$, $j = 1, \ldots , k$, as 
$$
\mathfrak{M}_{1}\otimes \dots \otimes \mathfrak{M}_{k}=\{M^{\lambda}_1\otimes \dots\otimes M^{\lambda}_k \;| \,\lambda\in \mathbb{R}_{+}\}. 
$$
Clearly, $\mathfrak{M}_{1}\otimes \dots \otimes \mathfrak{M}_{k}$ satisfies $[\operatorname{L}]$ ($[\mathfrak{M}.2]'$, respectively) if and only if each $\M_j$ does so.

We are interested in the ensuing special types of weight sequence systems. A weight sequence system $\M$ is called  \emph{isotropic} if, for each $\lambda \in \R_+$, $M^\lambda = (M^\lambda_{|\alpha|})_{\alpha\in \mathbb{N}^{d}}$ for a sequence $(M^\lambda_{p})_{p\in\mathbb{N}}$, which we identify with $M^\lambda$ itself.
Given a permutation $\sigma$ of the indices $\{1, \dots,d\}$, we write $\sigma(\mathfrak{M})=\{ (M^{\lambda}_{(\alpha_{\sigma(1)},\dots,\alpha_{\sigma(d)}) })_{(\alpha_1, \dots, \alpha_d)\in\N^{d}} \:|\:\lambda\in\mathbb{R}_{+}\}$. We call $\M$ \emph{isotropically decomposable} if there is a permutation $\sigma$ such that $\sigma(\M)=\mathfrak{M}_{1}\otimes \dots \otimes \mathfrak{M}_{k}$ with each $\M_{j}$ isotropic.  We also define these notions for single weight sequences in the natural way.

Given a weight sequence system $\M$, we associate to it the weight function system
$$
\W_\M = \{ e^{\omega_{M^\lambda}}  \, | \, \lambda \in \R_{+} \}.
$$ 	
If $\M$ is isotropically decomposable, various of the conditions on $\M$ and $\W_\M$ are related as follows.
\begin{lemma}\label{M2wss}
		Let $\M$ be an isotropically decomposable weight sequence system satisfying $[\operatorname{L}]$. 
		\begin{itemize}
		\item[$(a)$] $\W_{\M}$ satisfies $[\operatorname{M}]$.
		\item[$(b)$] Consider the following statements:
				\begin{itemize}
					\item[$(i)$] $\M$ satisfies $[\mathfrak{M}.2]'$.
					\item[$(ii)$] $A_{\W_{\M}}$ satisfies $[\operatorname{N}]$.
					\item[$(iii)$] $\W_{\M}$ satisfies $[\operatorname{N}]$.
				\end{itemize}
		\end{itemize}		
			Then, $(i) \Rightarrow (ii) \Leftrightarrow (iii)$. If in addition $\M$ is accelerating, then $(iii) \Rightarrow (i)$.
	\end{lemma}
	
	\begin{proof}  We may assume that $\M$ is isotropic, say,  $M^\lambda = (M^\lambda_{|\alpha|})_{\alpha\in \mathbb{N}^{d}}$ for a sequence $(M^\lambda_{p})_{p\in\mathbb{N}}$. We set 
$$
\eta^{\lambda}(t)=\sup_{p\in\mathbb{N}}\log \frac{t^{p}M_0^{\lambda}}{M^{\lambda}_p}, \qquad t\geq 0.
$$
Note that 
\begin{equation}
\eta^{\lambda}(d^{-1/2}|x|)\leq \omega_{M^{\lambda}}(x)\leq \eta^{\lambda}(|x|), \qquad x \in \R^d.
\label{equivalent}
\end{equation}

$(a)$ Condition $[\operatorname{L}]$ implies  that 	
		\begin{gather} 		
	\label{cond-L}			
					\forall R > 0 ~ \forall \lambda \in \R_{+} ~ \exists \mu \in \R_{+} ~ (\forall R > 0 ~ \forall \mu \in \R_{+} ~ \exists \lambda \in \R_{+}) ~ \exists C > 0 ~ \forall x \in \R^d \,: \\ \nonumber
					\omega_{M^{\lambda}}(R x) \leq \omega_{M^{\mu}}(x) + \log C.
				\end{gather}
Since $\eta^\lambda$ is increasing, \eqref{equivalent} implies that
$$
\omega_{M^\lambda}(x+y) \leq \omega_{M^\lambda}(2d^{1/2}x) +\omega_{M^\lambda}(2d^{1/2}y), \qquad x,y \in \R^d,
$$
the result follows from  \eqref{cond-L}.
		
		$(b)$ $(i) \Rightarrow (ii)$ This follows by combining  \eqref{cond-L} with the fact that  $[\mathfrak{M}.2]'$ implies  that
\begin{gather}
\label{cond-dc}
				\forall \lambda \in \R_{+} ~ \exists \mu \in \R_{+} ~ (\forall \mu \in \R_{+} ~ \exists \lambda \in \R_{+}) ~ \exists C',H' >0 ~ \forall x \in \R^d \, : \\ \nonumber
				\omega_{M^{\lambda}}(x) + \log (1+|x|)^{d}\leq  \omega_{M^{\mu}}(H'x) + \log C'. 
			\end{gather}
		
		$(ii) \Leftrightarrow (iii)$ In view of $(a)$, Lemma \ref{l:KotheNiffWN} yields the result.
		
		$(iii) \Rightarrow (i)$ \emph{(if $\M$ is accelerating)}  Define $m^{\lambda}_{p}= M^{\lambda}_{p}/M^{\lambda}_{p-1}$ for $p\geq1$ and let $m^{\lambda}(t)=\sum_{m^{\lambda}_{p}\leq t}1$ for $t \geq 0$. 
Set
$w^\lambda(x) = e^{\eta^{\lambda}(|x|)}$ for $x \in \R^d$. Note that the  weight function system  $\{ w^{\lambda}  \, | \, \lambda \in \R_{+} \}$ also satisfies $[N]$ because of \eqref{equivalent} and \eqref{cond-L}. 

It is well-known that  \cite[Equation (3.11)]{Komatsu} 
\[ 
w^{\lambda}(x) = \exp\left( \int_{0}^{|x|} \frac{m^{\lambda}(u)}{u} du
\right), \qquad x \in \R^d.
\]
Let $\lambda >0$ $(\mu >0)$  be arbitrary and choose  $\mu >0$ $(\lambda >0)$   such that $w^{\lambda}/w^{\mu} \in L^1$. In particular, $\mu \leq \lambda$. Since $\M$ is accelerating, we have that $m_{p}^{\mu} \leq m_{p}^{\lambda}$ for all $p \geq 1$ and thus $m^\lambda(t) \leq m^\mu(t)$ for all $t\geq 0$. Hence,			
\[\int_{t_1}^{t_2} \frac{m^{\lambda}(u) - m^{\mu}(u)}{u} du \leq 0\]
for all $t_2 \geq t_1 \geq 0$, which implies  that $w^\lambda(x)/w^\mu(x)$ is non-increasing in $|x|$ . Therefore, 
$$
|y|^d\frac{w^\lambda(y)}{w^\mu(y)}  \leq \frac{1}{|B(0,1)|}\int_{B(0,|y|)} \frac{w^\lambda(x)}{w^\mu(x)} dx \leq \frac{1}{|B(0,1)|} \int_{\R^d} \frac{w^\lambda(x)}{w^\mu(x)} dx  < \infty
$$ 
for all $y\in\R^{d}$. This implies \eqref{cond-dc} for $\{\eta^{\lambda}\:|\: \lambda\in\mathbb{R}_{+}\}$, namely,
\begin{gather*}
				\forall \lambda \in \R_{+} ~ \exists \mu \in \R_{+} ~ (\forall \mu \in \R_{+} ~ \exists \lambda \in \R_{+}) ~ \exists C',H' >0 \, \forall t \geq 0 \, : \, \\
				\eta^{\lambda}(t) + \log t \leq  \eta^{\mu}(H't) + \log C.
\end{gather*}
By combining the latter inequality with \cite[Proposition 3.2]{Komatsu}
\[
	M^{\lambda}_{p} = M_{0}^{\lambda} \sup_{t \geq 0} \frac{t^{p}}{e^{\eta^{\lambda}(t)}}, \qquad p \in \N,
\]
we obtain that $\M$ satisfies $[\M.2]'$.
\end{proof}

Finally, we present two examples of important instances of classes of weight sequence systems and weight function systems. Firstly, given a single weight sequence $M$ satisfying $(M.1)$, we  set $\M_{M} = \{ (\lambda^{|\alpha|} M_{\alpha})_{\alpha \in \N^{d}} \, | \, \lambda \in \R_{+} \}$ and $\W_{M} = \W_{\M_{M}} = \{ e^{\omega_M\left( \frac{\cdot}{\lambda} \right)} | \, \lambda \in \R_{+} \}$.
\begin{lemma} \label{lemma-1} Let $M$ be an isotropically decomposable weight sequence satisfying $(M.1)$. 
\begin{itemize} 
\item[$(a)$] $\M_{M}$ is accelerating and satisfies $[\operatorname{L}]$.
\item[$(b)$] $\W_{M}$ satisfies $[\operatorname{M}]$.
\item[$(c)$] $M$ satisfies $(M.2)'$ if and only if $\M_{M}$ satisfies $[\mathfrak{M}.2]'$  if and only if $\W_{M}$ satisfies $[\operatorname{N}]$.
\end{itemize}
\end{lemma}
\begin{proof}
Part $(a)$ is obvious, while $(b)$ and  $(c)$ have been established in Lemma \ref{M2wss}. 
\end{proof}
As a second example, following \cite[Section 5]{R-S-Composition}, we can also introduce weight sequence systems and weight function systems generated by a weight function in the sense of \cite{B-M-T}. We consider the following conditions on  a non-negative non-decreasing continuous function $\omega$ on  $[0,\infty)$:
	\begin{itemize}
		\item[$(\alpha)$] $\omega(2t) = O(\omega(t))$;
		\item[$(\gamma)$] $\log t = O(\omega(t))$;
		\item[$\{\gamma\}$] $\log t = o(\omega(t))$;
		\item[$(\delta)$] $\phi: [0, \infty) \rightarrow [0, \infty)$, $\phi(x) = \omega(e^{x})$, is convex.
	\end{itemize}
We call $\omega$ a  \emph{Braun-Meise-Taylor weight function (BMT weight function)} if $\omega_{|[0,1]} \equiv 0$ and $\omega$ satisfies $(\alpha)$, $\{\gamma\}$ and $(\delta)$.
In such a case, we define the \emph{Young conjugate} $\phi^{*}$ of $\phi$ as
	\[ \phi^{*} : [0, \infty) \rightarrow [0, \infty) , \qquad \phi^{*}(y) = \sup_{x \geq 0} (xy - \phi(x)). \]
Note that $\phi^{*}$ is convex and $y = o(\phi^{*}(y))$.
We define $\M_{\omega} = \{ M^{\lambda}_{\omega} \, | \,\lambda \in \R_{+}\}$, where $M^{\lambda}_{\omega} = \left(\exp\left(\frac{1}{\lambda} \phi^{*}(\lambda |\alpha| ) \right)\right)_{\alpha \in \N^{d}}$; the above stated properties of  $\phi^{*}$ imply that  $M^{\lambda}_{\omega}$ is an isotropic weight sequence satisfying $(M.1)$. Furthermore, we set $\W_{\omega} = \{ e^{\frac{1}{\lambda}\omega(| \, \cdot \, |)} \, | \, \lambda \in \R^{+} \}$ (for general $\omega$).
\begin{lemma} \label{lemma-2}
Let $\omega$ be a non-negative non-decreasing continuous function on  $[0,\infty)$. 
\begin{itemize}
\item[$(a)$]  If $\omega$ is a BMT weight function, then $\M_{\omega}$ satisfies $[\operatorname{L}]$ and $[\mathfrak{M}.2]'$.
\item[$(b)$] If $\omega$ satisfies $(\alpha)$, then $\W_{\omega}$ satisfies $[\operatorname{M}]$.
\item[$(c)$] $\omega$ satisfies $[\gamma]$ if and only if $\W_{\omega}$ satisfies $[\operatorname{N}]$.
\end{itemize}
\end{lemma}
\begin{proof}
$(a)$ This is shown in \cite[Corollary 5.15]{R-S-Composition}.

$(b)$ This follows from the fact that $\omega$ is non-decreasing.

$(c)$ As $\omega$ is non-decreasing, this can be shown by using a similar argument as in the  proof of  implication $(iii) \Rightarrow (i)$ in Lemma \ref{M2wss}$(b)$.
\end{proof}
\section{Gelfand-Shilov spaces}\label{sect-3}
 In this section, we define and discuss the Gelfand-Shilov spaces $\mathcal{S}^{[\mathfrak{M}]}_{[\mathscr{W}],q}$. Let $M = (M_\alpha)_{\alpha \in \N^d}$ be a sequence of positive numbers and let $w$ be a non-negative function on $\R^d$. We define $\mathcal{S}^{M}_{w,q}= \mathcal{S}^{M}_{w,q}(\R^d) $, $q \in [1,\infty]$, as the seminormed space consisting of all $\varphi \in C^\infty(\R^d)$ such that
$$
\| \varphi \|_{\mathcal{S}^{M}_{w,q}} = \sup_{\alpha \in \N^d} \frac{1}{M_{\alpha}} \left (\int_{\R^d} (|\varphi^{(\alpha)}(x)| w(x))^q dx \right)^{1/q} < \infty, \qquad q \in [1,\infty),
$$
and
$$
\| \varphi \|_{\mathcal{S}^{M}_{w,\infty}} = \sup_{\alpha \in \N^d} \sup_{x \in \R^d} \frac{|\varphi^{(\alpha)}(x)| w(x) }{M_{\alpha}}< \infty. 
$$
If $w$ is positive and continuous, then  $\mathcal{S}^{M}_{w,q}$ is a Banach space. Given a weight sequence system $\mathfrak{M}$ and a weight function system $\mathscr{W}$, we define the \emph{Gelfand-Shilov spaces (of Beurling and Roumieu type)}
$$
\mathcal{S}^{(\mathfrak{M})}_{(\mathscr{W}),q} = \varprojlim_{\lambda \to 0^+} \mathcal{S}^{M^\lambda}_{w^\lambda,q}, \qquad \mathcal{S}^{\{\mathfrak{M}\}}_{\{\mathscr{W}\},q}= \varinjlim_{\lambda \to \infty} \mathcal{S}^{M^\lambda}_{w^\lambda,q}, \qquad q \in [1,\infty].
$$
Note that $\mathcal{S}^{(\mathfrak{M})}_{(\mathscr{W}),q}$ is a Fr\'echet space, while $\mathcal{S}^{\{\mathfrak{M}\}}_{\{\mathscr{W}\},q}$ is an $(LB)$-space. If $\W$ satisfies $[\operatorname{wM}]$, then $\mathcal{S}^{[\mathfrak{M}]}_{[\mathscr{W}],q}$ is translation-invariant; we shall tacitly use this fact in the sequel. Given two weight sequences $M$ and $A$, we define $\mathcal{S}^{[M]}_{[A],q} = \mathcal{S}^{[\mathfrak{M}_M]}_{[\mathscr{W}_A],q}$. Similarly, given a BMT weight function $\omega$ and a non-negative non-decreasing continuous function $\eta$ on $[0,\infty)$, we set $\mathcal{S}^{[\omega]}_{[\eta],q} = \mathcal{S}^{[\mathfrak{M}_\omega]}_{[\mathscr{W}_\eta],q}$. 
\begin{lemma}\label{regular}
Let $\mathfrak{M}$ be a weight sequence system, let $\mathscr{W}$ be a  weight function system and let $q \in [1,\infty]$. Then, the $(LB)$-space $\mathcal{S}^{\{\mathfrak{M}\}}_{\{\mathscr{W}\},q}$ is  regular.
\end{lemma}
\begin{proof}
By \cite[p.\ 80, Corollary 7]{Bierstedt}, it suffices to show that, for each $\lambda >0$, the closed unit ball $B_\lambda$ in $\mathcal{S}^{M^\lambda}_{w^\lambda,q}$ is closed in $\mathcal{S}^{\{\mathfrak{M}\}}_{\{\mathscr{W}\},q}$. Note that $\mathcal{S}^{\{\mathfrak{M}\}}_{\{\mathscr{W}\},q} \subset \mathcal{D}_{L^q} \subset \mathcal{B}$ with continuous inclusion; the first inclusion is a consequence of the fact that $1 \leq w^\lambda$ for all $\lambda >0$ and the second one is a classical result of Schwartz \cite{Schwartz}. Therefore, it is enough to prove that $B_\lambda$ is closed in $\mathcal{B}$. Let $(\varphi_n)_{n \in \N}$ be a sequence in $B_\lambda$ and $\varphi \in \mathcal{B}$ such that $\varphi_n \rightarrow \varphi$ in $\mathcal{B}$. In particular, $\varphi_n^{(\alpha)}(x) \to \varphi^{(\alpha)}(x)$ for all $\alpha \in \N^d$ and $x \in \R^d$. Hence, we obtain that
$$
\| \varphi^{(\alpha)} w_\lambda\|_{L^q} \leq \liminf_{n \to \infty} \|\varphi_n^{(\alpha)} w^{\lambda}\|_{L^q} \leq M^\lambda_{\alpha}
$$
for all $\alpha \in \N^d$, where we have used  Fatou's lemma for $q < \infty$. This shows that $\varphi \in B_\lambda$ and the proof is complete.
\end{proof}

We now study when the equality  $\mathcal{S}^{[\mathfrak{M}]}_{[\mathscr{W}],q} = \mathcal{S}^{[\mathfrak{M}]}_{[\mathscr{W}],r}$ holds.
\begin{theorem}\label{ft-char-N}
Let $\mathfrak{M}$ be a weight sequence system satisfying $[\operatorname{L}]$ and $[\mathfrak{M}.2]'$, and let $\mathscr{W}$ be a weight function system satisfying $[\operatorname{wM}]$. Suppose that  $\mathcal{S}^{[\mathfrak{M}]}_{[\mathscr{W}],q} \neq \{0\}$ for some $q \in [1,\infty]$. Consider the following statements:
\begin{itemize}
\item[$(i)$] $\mathscr{W}$ satisfies $[\operatorname{N}]$.
\item[$(ii)$] $\mathcal{S}^{[\mathfrak{M}]}_{[\mathscr{W}],q} = \mathcal{S}^{[\mathfrak{M}]}_{[\mathscr{W}],r}$ as locally convex spaces for all $q,r \in  [1,\infty]$.
\item[$(iii)$] $\mathcal{S}^{[\mathfrak{M}]}_{[\mathscr{W}],q} = \mathcal{S}^{[\mathfrak{M}]}_{[\mathscr{W}],r}$ as sets for some  $q,r \in  [1,\infty]$ with $q \neq r$.
\end{itemize}
Then, $(i) \Rightarrow (ii) \Rightarrow (iii)$. If  in addition $\mathscr{W}$ satisfies $[\operatorname{M}]$, then also $(iii) \Rightarrow (i)$.
\end{theorem}

We need several results in preparation for the proof of Theorem \ref{ft-char-N}.  

\begin{lemma} \label{reverse-inclusion}
Let $\mathfrak{M}$ be a weight sequence system  and let  $\mathscr{W}$ be a weight function system satisfying $[\operatorname{wM}]$ and $[\operatorname{N}]$. Then, $\mathcal{S}^{[\mathfrak{M}]}_{[\mathscr{W}],q} \subseteq \mathcal{S}^{[\mathfrak{M}]}_{[\mathscr{W}],r}$ with continuous inclusion for all $q,r \in  [1,\infty]$ with $q \geq r$. 
\end{lemma}
\begin{proof}
This follows from H\"older's inequality and Lemma \ref{lemma-C0}. 
\end{proof}
Given a weight sequence system $\mathfrak{M}$ and a weight function system $\mathscr{W}$, we introduce the auxiliary spaces
$$
\widetilde{\mathcal{S}}^{(\mathfrak{M})}_{(\mathscr{W})} = \bigcap_{\lambda > 0} \bigcap_{k \in \N} \mathcal{S}^{M^\lambda}_{(1+|\,\cdot\,|)^k w^\lambda,\infty}, \qquad \widetilde{\mathcal{S}}^{\{\mathfrak{M}\}}_{\{\mathscr{W}\}} = \bigcup_{\lambda > 0} \bigcap_{k \in \N} \mathcal{S}^{M^\lambda}_{(1+|\,\cdot\,|)^k w^\lambda,\infty}.
$$
\begin{lemma}\label{non-trivial}
Let $\mathfrak{M}$ be a weight sequence system satisfying $[\operatorname{L}]$ and let $\mathscr{W}$ be a weight function system satisfying $[\operatorname{wM}]$. The following statements are equivalent:
\begin{itemize}
\item[$(i)$] $\mathcal{S}^{[\mathfrak{M}]}_{[\mathscr{W}],q} \neq \{0\}$ for all $q \in [1,\infty]$.
\item[$(ii)$] $\mathcal{S}^{[\mathfrak{M}]}_{[\mathscr{W}],q} \neq \{0\}$ for some $q \in [1,\infty]$.
\item[$(iii)$] $\widetilde{\mathcal{S}}^{[\mathfrak{M}]}_{[\mathscr{W}]} \neq \{0\}$.
\end{itemize}
\end{lemma}
\begin{proof}
$(i) \Rightarrow (ii)$ Trivial.

$(ii) \Rightarrow (iii)$ Let $\varphi \in \mathcal{S}^{[\mathfrak{M}]}_{[\mathscr{W}],q}$ be such that $\varphi(0) = 1$. Choose $\psi \in \mathcal{D}(\R^d)$ such that $\int_{\R^d} \varphi(x) \psi(-x) dx =1$. Next, pick $\chi \in \mathcal{D}(\R^d)$ such that $\int_{\R^d} \chi (x)dx = 1$ and consider its Fourier transform $\widehat{\chi}(\xi) = \int_{\R^d} \chi(x)e^{-2\pi i \xi \cdot x} dx$. Then, $\varphi_0 = (\varphi \ast \psi) \widehat{\chi} \in \widetilde{\mathcal{S}}^{[\mathfrak{M}]}_{[\mathscr{W}]}$ and $\varphi_0 \not \equiv 0$ (as $\varphi_0(0) = 1$).

$(iii) \Rightarrow (i)$ This follows from the fact that $\widetilde{\mathcal{S}}^{[\mathfrak{M}]}_{[\mathscr{W}]} \subset \mathcal{S}^{[\mathfrak{M}]}_{[\mathscr{W}],q}$ for all $q \in [1,\infty]$.
\end{proof}
Next, we establish an important connection between the spaces  $\mathcal{S}^{[\mathfrak{M}]}_{[\mathscr{W}],q}$ and $\lambda^q[A_\mathscr{W}]$.
\begin{proposition}\label{compl-1}
Let $\mathfrak{M}$ be a weight sequence system, let $\mathscr{W}$ be a weight function system satisfying $[\operatorname{wM}]$ and let $q \in [1,\infty]$. The mapping
$$
S_q = S: \mathcal{S}^{[\mathfrak{M}]}_{[\mathscr{W}],q} \rightarrow \lambda^q[A_\mathscr{W}], \qquad S(\varphi) = (\varphi(j))_{j \in \Z^d}
$$
is continuous.
\end{proposition}
\begin{proof}
For $q = \infty$ this is obvious. Assume now that $q < \infty$. We denote by $H$ the characteristic function of the orthant $[0,\infty)^d$ and set $\partial = \partial_d \cdots \partial_1$. Then, $\partial H = \delta$. Choose $\psi \in \mathcal{D}_{\left[-\frac{1}{2}, \frac{1}{2}\right]^d}$ such that $\psi \equiv 1$ on a neighbourhood of $0$. Then, $\partial(H\psi) - \delta = \chi \in L^{\infty}$ and $\operatorname*{supp}\chi \subset \left[-\frac{1}{2}, \frac{1}{2}\right]^d$. Hence, $\varphi = (\partial \varphi) \ast (H\psi) - \varphi \ast \chi$ for all $\varphi \in C^\infty(\R^d)$. For each $\lambda > 0$ there are $\mu > 0$ and $C > 0$ (for each $\mu > 0$ there are $\lambda > 0$ and $C > 0$) such that $w^\lambda(x+t) \leq Cw^\mu(x)$ for all $x \in \R^d$ and $t \in [-\frac{1}{2}, \frac{1}{2}]^d$. We obtain that 
$$
|\varphi(x) w^\lambda(x)| \leq C\left(\|\psi\|_{L^\infty} \int_{x + [-\frac{1}{2}, \frac{1}{2}]^d} |\partial \varphi(t)| w^\mu(t) dt + \|\chi\|_{L^\infty} \int_{x+ [-\frac{1}{2}, \frac{1}{2}]^d} |\varphi(t)| w^\mu(t) dt\right)
$$
for all $x \in \R^d$ and $\varphi \in C^\infty(\R^d)$.  By Jensen's inequality, the latter inequality implies that
$$
\| (\varphi(j)w^\lambda(j))_{j \in \Z^d} \|_{l^q} \leq C(\|\psi\|_{L^\infty} \| \partial \varphi w^\mu\|_{L^q} + \|\chi\|_{L^\infty} \|  \varphi w^\mu\|_{L^q})
$$
for all $\varphi \in \mathcal{S}^{M^\mu}_{w^\mu, q}$, whence the result follows.
\end{proof}

\begin{proposition}\label{compl-2}
Let $\mathfrak{M}$ be a weight sequence system, let $\mathscr{W}$ be a weight function system satisfying $[\operatorname{M}]$ and let $q \in [1,\infty]$. For each $\psi \in \widetilde{\mathcal{S}}^{[\mathfrak{M}]}_{[\mathscr{W}]}$, the mapping 
$$
T_{\psi,q} = T_\psi = T:  \lambda^q[A_\mathscr{W}] \rightarrow \mathcal{S}^{[\mathfrak{M}]}_{[\mathscr{W}],q}, \qquad T( (c_j)_{j \in \Z^d}) = \sum_{j \in \Z^d} c_j \psi (\, \cdot - j)
$$
is continuous.
\end{proposition}
\begin{proof}
We only show the result for $q \in (1,\infty)$; the proofs for  $q = 1$ and $q = \infty$ are similar and in fact simpler. Let $\nu > 0$ be such that $\psi \in \bigcap_{k \in \N} \mathcal{S}^{M^\nu}_{(1+|\,\cdot\,|)^kw^\nu,\infty}$; this means that $\nu$  is fixed in the Roumieu case but can be taken as small as needed in the Beurling case. For each $\lambda > 0$ there are $\mu, \nu > 0$ and $C > 0$ (for each $\mu > 0$ there are $\lambda > 0$ and $C > 0$) such that $w^\lambda(x+y) \leq Cw^\mu(x)w^\nu(y)$ for all $x,y \in \R^d$. We may assume that $\nu \leq \lambda$. Let $q' = q/(q-1)$ be the conjugate exponent of $q$. By H\"older's inequality, we have that, for all $(c_j)_{j \in \Z^d} \in l^q((w^{\mu}(j))_{j \in \Z^d})$,
\begin{align*}
&\sum_{j \in \Z^d} |c_j| |\psi^{(\alpha)}(x-j)| w^\lambda(x) \\
&\leq C \sum_{j \in \Z^d} \frac{|c_j| w^{\mu}(j)}{(1+|x-j|)^{(d+1)/q}} |\psi^{(\alpha)}(x-j)| w^\nu(x-j)(1+|x-j|)^{(d+1)/q} \\
&\leq C  \left(\sum_{j \in \Z^d} \frac{(|c_j| w^{\mu}(j))^q}{(1+|x-j|)^{d+1}}\right)^{1/q}  \times \\
&\phantom{\leq} \left(\sum_{j \in \Z^d} \left(|\psi^{(\alpha)}(x-j)| w^\nu(x-j)(1+|x-j|)^{(d+1)/q)} \right)^{q'}\right)^{1/q'} \\
&\leq C' \|\psi^{(\alpha)} (1+| \, \cdot \, |)^{d+1}w^\nu\|_{L^\infty} \left(\sum_{j \in \Z^d} \frac{(|c_j| w^{\mu}(j))^q}{(1+|x-j|)^{d+1}}\right)^{1/q}  
\end{align*}
for all $\alpha \in \N^d$ and $x \in \R^d$, where
$
C' = 2^{\frac{d+1}{q'}} C\left(\sum_{j \in \Z^d} (1+|j|)^{-d-1}\right)^{1/q'}. 
$
Hence,
\begin{align*}
 \| \sum_{j \in \Z^d} c_j \psi (\, \cdot - j) \|_{\mathcal{S}^{M^\lambda}_{w^\lambda,q}} &\leq \sup_{\alpha \in \N^d}\frac{1}{M^\nu_{\alpha}}  \| \sum_{j \in \Z^d} c_j \psi^{(\alpha)} (\, \cdot - j) w^\lambda \|_{L^q} \\
& \leq C'' \|\psi\|_{ \mathcal{S}^{M^\nu}_{(1+|\,\cdot\,|)^{d+1}w^\nu,\infty}} \|(c_jw^\mu(j))_{j \in \Z^d}\|_{l^q},
\end{align*}
with
$
C'' = C'\left(\int_{x \in \R^d} (1+|x|)^{-d-1} dx \right)^{1/q}. 
$
\end{proof}
\begin{lemma}\label{construction}
Let $\mathfrak{M}$ be a weight sequence system satisfying $[\operatorname{L}]$ and let $\mathscr{W}$ be a weight function system satisfying $[\operatorname{wM}]$. Suppose that  $\mathcal{S}^{[\mathfrak{M}]}_{[\mathscr{W}],q} \neq \{0\}$ for some $q \in [1,\infty]$. 
\begin{itemize}
\item[$(a)$]There exists $\psi \in  \widetilde{\mathcal{S}}^{[\mathfrak{M}]}_{[\mathscr{W}]}$ such that $\psi(j) = \delta_{j,0}$ for all $j \in \Z^d$.
\item[$(b)$] There exists  $\psi \in  \widetilde{\mathcal{S}}^{[\mathfrak{M}]}_{[\mathscr{W}]}$ such that $\sum_{j \in \Z^d} \psi(\,\cdot - j) \equiv 1$.
\end{itemize}
\end{lemma}
\begin{proof}
$(a)$ By Lemma \ref{non-trivial}, there exists $\varphi \in  \widetilde{\mathcal{S}}^{[\mathfrak{M}]}_{[\mathscr{W}]}$ such that $\varphi(0) = 1$. Set 
$$
\chi(\xi) = \mathcal{F}(1_{[-\frac{1}{2}, \frac{1}{2}]^d})(\xi) = \int_{[-\frac{1}{2}, \frac{1}{2}]^d} e^{-2\pi i\xi \cdot x} dx, \qquad \xi \in \R^d.
$$
Then, $\chi(j) = \delta_{j,0}$ for all $j \in \Z^d$. Hence, $\psi= \varphi \chi$ satisfies all requirements. 

$(b)$ By Lemma \ref{non-trivial},  there is $\varphi \in   \widetilde{\mathcal{S}}^{[\mathfrak{M}]}_{[\mathscr{W}]}$ such that $\int_{\R^d} \varphi(x) dx = 1$  . Then,
$$
\psi(x) = \int_{[-\frac{1}{2}, \frac{1}{2}]^d} \varphi(x-t) dt, \qquad x \in \R^d,
$$
satisfies all requirements.
\end{proof}
We obtain the following useful corollary. 
\begin{corollary}\label{compl}
Let $\mathfrak{M}$ be a weight sequence system satisfying $[\operatorname{L}]$, let $\mathscr{W}$ be a weight function system satisfying $[\operatorname{M}]$ and let $q \in [1,\infty]$. Suppose that  $\mathcal{S}^{[\mathfrak{M}]}_{[\mathscr{W}],q} \neq \{0\}$. Then, $\lambda^q[A_\mathscr{W}]$ is isomorphic to a complemented subspace of  $\mathcal{S}^{[\mathfrak{M}]}_{[\mathscr{W}],q}$.
\end{corollary}
\begin{proof}
Choose $\psi$ as in Lemma \ref{construction}$(a)$. Consider the continuous linear mappings $S: \mathcal{S}^{[\mathfrak{M}]}_{[\mathscr{W}],q} \rightarrow \lambda^q[A_\mathscr{W}]$ and $T_\psi:  \lambda^q[A_\mathscr{W}] \rightarrow \mathcal{S}^{[\mathfrak{M}]}_{[\mathscr{W}],q}$ from Proposition \ref{compl-1} and Proposition \ref{compl-2}, respectively, and note that $S \circ T_\psi = \operatorname{id}_{\lambda^q[A_\mathscr{W}]}$.
\end{proof}

\begin{proof}[Proof of Theorem \ref{ft-char-N}]
$(i) \Rightarrow (ii)$ By Lemma \ref{reverse-inclusion}, it suffices to show that $\mathcal{S}^{[\mathfrak{M}]}_{[\mathscr{W}],q} \subseteq \mathcal{S}^{[\mathfrak{M}]}_{[\mathscr{W}],r}$ with continuous inclusion for all  $q \leq r$. Since $\|f\|_{L^r} \leq \|f\|^{(r-q)/r}_{L^\infty} \|f\|^{q/r}_{L^q}$ for all $f \in L^\infty \cap L^r$, it is enough to consider the case $r = \infty$.  
 We use the same notation as in the proof of Proposition \ref{compl-1}. By $[\operatorname{wM}]$, $[\mathfrak{M}.2]'$ and $[L]$, we find that for each $\lambda > 0$ there are $\mu > 0$ and $C,C' > 0$ (for each $\mu > 0$ there are $\lambda > 0$ and $C,C' > 0$) such that
$w^\lambda(x+t) \leq Cw^\mu(x)$ for all $x \in \R^d$ and $t \in [-\frac{1}{2}, \frac{1}{2}]^d$ and $M^\mu_{\alpha+e} \leq C'M^\lambda_{\alpha}$ for all $\alpha  \in \N^{d}$, where $e=(1,1,\dots, 1)$ . We may assume that $\mu \leq \lambda$. Hence, by Jensen's inequality,
\begin{align*}
\| \varphi \|_{\mathcal{S}^{M^\lambda}_{w^\lambda,\infty}} &= \sup_{\alpha \in \N^d} \sup_{x \in \R^d}  \frac{1}{M^\lambda_{\alpha}} |\varphi^{(\alpha)}(x)| w^\lambda(x) \\
&\leq  C\|\psi\|_{L^\infty}\sup_{\alpha \in \N^d} \sup_{x \in \R^d}\frac{1}{M^\lambda_{\alpha}} \int_{x + [-\frac{1}{2}, \frac{1}{2}]^d} |\partial \varphi^{(\alpha)}(t)| w^\mu(t) dt + \\
&\phantom{\leq \,\, \, } C\|\chi\|_{L^\infty}\sup_{\alpha \in \N^d} \sup_{x \in \R^d} \frac{1}{M^\lambda_{\alpha}}  \int_{x+ [-\frac{1}{2}, \frac{1}{2}]^d} |\varphi^{(\alpha)}(t)| w^\mu(t) dt \\
&\leq  CC'\|\psi\|_{L^\infty}\sup_{\alpha \in \N^d} \sup_{x \in \R^d}\frac{1}{M^\mu_{\alpha +e}} \left(\int_{x + [-\frac{1}{2}, \frac{1}{2}]^d} (|\partial \varphi^{(\alpha)}(t)| w^\mu(t))^q dt\right)^{1/q} + \\
&\phantom{\leq \, \, \, } C\|\chi\|_{L^\infty}\sup_{\alpha \in \N^d} \sup_{x \in \R^d} \frac{1}{M^\mu_{\alpha}}  \left(\int_{x+ [-\frac{1}{2}, \frac{1}{2}]^d} (|\varphi^{(\alpha)}(t)| w^\mu(t))^q dt \right)^{1/q} \\
&\leq C'' \| \varphi \|_{\mathcal{S}^{M^\mu}_{w^\mu,q}},
\end{align*}
for all $\varphi \in  \mathcal{S}^{M^\mu}_{w^\mu,q}$, where $C'' = C(C'\|\psi\|_{L^\infty} + \|\chi\|_{L^\infty})$. 

$(ii) \Rightarrow (iii)$ Trivial.

$(iii) \Rightarrow (i)$ \emph{(if $\mathscr{W}$ satisfies $[\operatorname{M}]$)} Suppose that $q < r$. Choose $\psi$ as in Lemma \ref{construction}$(a)$. Consider the mappings $S_q: \mathcal{S}^{[\mathfrak{M}]}_{[\mathscr{W}],q} \rightarrow \lambda^q[A_\mathscr{W}]$ and $T_{\psi,r}:  \lambda^r[A_\mathscr{W}] \rightarrow \mathcal{S}^{[\mathfrak{M}]}_{[\mathscr{W}],r}$ from Proposition \ref{compl-1} and Proposition \ref{compl-2}, respectively. Note that $c = S_q (T_{\psi,r}(c)) \in \lambda^q[A_\mathscr{W}]$ for all $c \in \lambda^r[A_\mathscr{W}]$, that is, $\lambda^r[A_\mathscr{W}] \subseteq \lambda^q[A_\mathscr{W}]$. Since $\lambda^q[A_\mathscr{W}] \subseteq \lambda^r[A_\mathscr{W}]$ always holds true, we have that $\lambda^r[A_\mathscr{W}] = \lambda^q[A_\mathscr{W}]$ as sets. The result now follows from Proposition \ref{nuclear-echelon} and Lemma \ref{l:KotheNiffWN}.
\end{proof}
In the sequel, we shall often drop the index $q$ in the notation $\mathcal{S}^{[\mathfrak{M}]}_{[\mathscr{W}],q}$ if $\mathfrak{M}$ is a weight sequence system satisfying $[\operatorname{L}]$ and $[\mathfrak{M}.2]'$ and $\mathscr{W}$ is a weight function system satisfying $[\operatorname{wM}]$ and $[N]$. This is justified by Theorem \ref{ft-char-N}.
\section{Nuclearity}\label{sect-4}
In this main section, we characterize the nuclearity of the Gelfand-Shilov spaces $\mathcal{S}^{[\mathfrak{M}]}_{[\mathscr{W}],q}$ in terms of $\mathfrak{M}$ and $\mathscr{W}$. We start with the following result.
\begin{theorem}\label{ultra-suff}
Let $\mathfrak{M}$ be a weight sequence system satisfying $[\operatorname{L}]$ and $[\mathfrak{M}.2]'$ and let $\mathscr{W}$ be a weight function system satisfying $[\operatorname{wM}]$ and $[\operatorname{N}]$. Then,  $\mathcal{S}^{[\mathfrak{M}]}_{[\mathscr{W}]}$ is nuclear.
\end{theorem}
Our proof of Theorem \ref{ultra-suff} is based on Grothendieck's criterion  for nuclearity in terms of summable sequences \cite{Grothendieck}. Let $E$ be a lcHs (=Hausdorff locally convex space) and denote by $\csn(E)$ the set of all continuous seminorms on $E$. A sequence $(e_n)_{n \in \N}$ is called \emph{weakly summable} if $\sum_{n=0}^\infty |\langle e',e_n \rangle | < \infty$ for all $e' \in E'$. By Mackey's theorem, $(e_n)_{n \in \N}$ is weakly summable if and only if the set
$$
\bigcup_{k \in \N}\left \{ \sum_{n = 0}^k c_n e_n \, | \,  |c_n| \leq 1, n = 0, \ldots, k \right\}
$$
is bounded in $E$.  The sequence $(e_n)_{n \in \N}$ is called \emph{absolutely summable} if $\sum_{n=0}^\infty  p(e_n)  < \infty$ for all $p \in \csn(E)$. Clearly, $(e_n)_{n \in \N}$ is absolutely summable if and only if $\sum_{n=0}^\infty p(e_n)< \infty$ for all $p$ belonging to some fundamental system of continuous seminorms on $E$. 

\begin{proposition}[{\cite[Theorem 4.2.5]{Pietsch}}] \label{summable}
Let $E$ be a Fr\'echet space or a $(DF)$-space. Then, $E$ is nuclear if and only if every weakly summable sequence in $E$ is absolutely summable.
\end{proposition}
\begin{proof}[Proof of Theorem \ref{ultra-suff}]
We shall show that $\mathcal{S}^{[\mathfrak{M}]}_{[\mathscr{W}]} = \mathcal{S}^{[\mathfrak{M}]}_{[\mathscr{W}],\infty}$ is nuclear. To this end, we employ Proposition \ref{summable} with $E = \mathcal{S}^{[\mathfrak{M}]}_{[\mathscr{W}],\infty}$. Let $(\varphi_n)_{n \in \N} \subset \mathcal{S}^{[\mathfrak{M}]}_{[\mathscr{W}],\infty}$ be a weakly summable sequence. This means that for all $\lambda > 0$ (for some $\lambda > 0$) there is $C > 0$ such that 
$$
\left \| \sum_{n = 0}^k c_n \varphi_n \right \|_{\mathcal{S}^{M^\lambda}_{w^\lambda,\infty}} \leq C
$$
for all  $k \in \N$ and $|c_n| \leq 1$,  $n = 0, \ldots, k$, where we have used Lemma \ref{regular} in the Roumieu case. We claim that
\begin{equation}
\sup_{\alpha \in \N^d} \sup_{x \in \R^d} \frac{1}{M^\lambda_{\alpha}}\sum_{n=0}^\infty |\varphi_n^{(\alpha)}(x)| w^\lambda(x) \leq C.
\label{pointwise}
\end{equation}
Fix arbitrary $\alpha \in \N^d$ and $x \in \R^d$. Let $k \in \N$. Choose $|c_n(\alpha,x)| \leq 1$ such that $c_n(\alpha,x)\varphi_n^{(\alpha)}(x) = |\varphi_n^{(\alpha)}(x)|$. Then,
$$
\frac{1}{M^\lambda_{\alpha}}\sum_{n=0}^k |\varphi_n^{(\alpha)}(x)| w^\lambda(x) = \frac{1}{M^\lambda_{\alpha}}\left| \sum_{n=0}^k c_n(\alpha,x)\varphi_n^{(\alpha)}(x) \right|w^\lambda(x) \leq C,
$$
whence the claim follows by letting $k \to \infty$. We now employ \eqref{pointwise} to show that $(\varphi_n)_{n \in \N}$ is absolutely summable. By Theorem \ref{ft-char-N}, it is enough to prove that  
$$
 \sum_{n = 0}^\infty \| \varphi_n  \|_{\mathcal{S}^{M^\mu}_{w^\mu,1}}  < \infty
$$
for all $\mu > 0$ (for some $\mu > 0$). Let $\mu > 0$ be arbitrary (let $\lambda > 0$ be  such that \eqref{pointwise} holds). Conditions $[\operatorname{L}]$ and $[\operatorname{N}]$ imply that there
is $\lambda > 0$ (there is $\mu > 0$) such that $2^{|\alpha|}M^\lambda_\alpha \leq C'M^\mu_\alpha$ for all $\alpha \in \N^{d}$ and some $C' > 0$ and $w^\mu/w^\lambda \in L^1$. Hence,
\begin{align*}
 \sum_{n = 0}^\infty \| \varphi_n  \|_{\mathcal{S}^{M^\mu}_{w^\mu,1}} &=  \sum_{n = 0}^\infty \sup_{\alpha \in \N^d}  \frac{1}{M^\mu_{\alpha}} \int_{\R^d} |\varphi^{(\alpha)}_n(x)| w^\mu(x) dx \\
 &\leq C'  \sum_{\alpha \in \N^d} \frac{1}{2^{|\alpha|}}   \int_{\R^d} \frac{1}{M^\lambda_{\alpha}}\sum_{n = 0}^\infty  |\varphi^{(\alpha)}_n(x)| w^\lambda(x)  \frac{w^\mu(x)}{w^\lambda(x)} dx \\
 &\leq 2^dCC' \|w^\mu/w^\lambda \|_{L^1}.
 \end{align*}
\end{proof}

Our next goal is to discuss the necessity of the conditions $[\mathfrak{M}.2]'$ and $[\operatorname{N}]$ for $\mathcal{S}^{[\mathfrak{M}]}_{[\mathscr{W}],q}(\R^d)$ to be nuclear.
\begin{proposition}\label{ultra-necc-1}
Let $\mathfrak{M}$ be a weight sequence system satisfying $[\operatorname{L}]$, let $\mathscr{W}$ be a weight function system satisfying $[\operatorname{M}]$ and let $q \in [1,\infty]$. Suppose that $\mathcal{S}^{[\mathfrak{M}]}_{[\mathscr{W}],q}(\R^d)$ is non-trivial and nuclear. Then, $\mathscr{W}$ satisfies $[\operatorname{N}]$.
\end{proposition}
\begin{proof} 
Since nuclearity is inherited to subspaces, Corollary \ref{compl} implies that $\lambda^q[A_{\mathscr{W}}]$ is nuclear. The result therefore follows from Proposition \ref{nuclear-echelon} and Lemma \ref{l:KotheNiffWN}.

\end{proof}
\begin{proposition}\label{ultra-necc-2}
Let $\mathfrak{M}$ be a weight sequence system satisfying $[\operatorname{L}]$, let $\mathscr{W}$ be a weight function system satisfying $[\operatorname{M}]$ and let $q \in [1,\infty]$. Suppose that $\mathcal{S}^{[\mathfrak{M}]}_{[\mathscr{W}],q}(\R^d)$ is non-trivial and nuclear. Then, $A_{\mathscr{W}_\mathfrak{M}}$ satisfies $[\operatorname{N}]$.
\end{proposition}

We shall make use of the ensuing result due to Petzsche \cite{Petzsche} in order to show Proposition \ref{ultra-necc-2}.

\begin{lemma} [{\cite[Satz 3.5 and Satz 3.6]{Petzsche}}] \label{P-trick} Let $A$ be a K\"othe set and let $E$ be a lcHs.
\begin{itemize}
\item[$(a)$]Suppose that $E$ is nuclear and that there are continuous linear mappings $T:\lambda^1(A) \rightarrow E$ and $S: E \rightarrow \lambda^\infty(A)$ such that $S \circ T = \iota$, where $\iota: \lambda^1(A) \rightarrow \lambda^\infty(A)$ denotes the natural embedding. Then, $\lambda^1(A)$ is nuclear. 
\item[$(b)$]Suppose that $E'_\beta$ is nuclear and that there are continuous linear mappings $T:\lambda^1\{A\} \rightarrow E$ and $S: E \rightarrow \lambda^\infty\{A\}$ such that $S \circ T = \iota$, where $\iota: \lambda^1\{A\} \rightarrow \lambda^\infty\{A\}$ denotes the natural embedding. Then, $\lambda^1\{A\}$ is nuclear. 
\end{itemize}
\end{lemma}
\begin{proof}
This is essentially shown in \cite[Satz 3.5 and Satz 3.6]{Petzsche} but we repeat the argument here for the sake of completeness and because our assumptions are slightly more general. 

$(a)$ Since nuclearity is inherited to subspaces, it suffices to show that $T$ is a topological isomorphism onto its image. We write $e_i = (\delta_{j,i})_{j \in \Z^d}$ for  $i \in \Z^d$. Then, $(e_i)_{i \in \Z^d}$ is a Schauder basis for  $\lambda^1(A)$ with coefficient functionals
$$
\xi_i: \lambda^1(A) \rightarrow \C, \quad \langle \xi_i,  (c_j)_{j \in \Z^d}\rangle = c_i, \qquad i \in \Z^d.
$$
Since $T$ is continuous and $S \circ T = \iota$,  $(T(e_i))_{i \in \Z^d}$ is a Schauder basis for  $T(\lambda^1(A))$ with coefficient functionals $\eta_i = \xi_i \circ T^{-1} = \xi_i \circ S$ for $i \in \Z^d$. We claim that the Schauder basis $(T(e_i))_{i \in \Z^d}$ is equicontinuous, that is,
$$
\forall p \in \operatorname{csn}(E) \, \exists q \in \operatorname{csn}(E) \, \forall x \in T(\lambda^1(A)) \, : \, \sup_{i \in \Z^d} |  \langle \eta_i, x \rangle | p(T(e_i)) \leq q(x).
$$
Let $p \in \operatorname{csn}(E)$ be arbitrary. As $T$ is continuous, there is $\lambda > 0$ such that
$$
|  \langle \eta_i, x \rangle | p(T(e_i)) \leq  |  \langle \xi_i, S(x) \rangle | \|e_i\|_{l^1(a^\lambda)} = |\langle \xi_i, S(x) \rangle | a^\lambda_{i} \leq \| S(x) \|_{l^\infty(a^\lambda)}
$$
for all $x \in T(\lambda^1(A))$ and $i \in \Z^d$. The claim now follows from the continuity of $S$.  Since $T(\lambda^1(A))$ is nuclear (as a subspace of the nuclear space $E$), the Dymin-Mityagin basis theorem \cite[Theorem 10.2.1]{Pietsch} yields that $(T(e_i))_{i \in \Z^d}$ is an absolute Schauder basis for $T(\lambda^1(A))$, that is,
\begin{equation}
\forall p \in \operatorname{csn}(E) \, \exists q \in \operatorname{csn}(E) \, \forall x \in T(\lambda^1(A)) \, : \, \sum_{i \in \Z^d} |  \langle \eta_i, x \rangle | p(T(e_i)) \leq q(x).
\label{absolute}
\end{equation}
We now show that $T^{-1}: T(\lambda^1(A)) \rightarrow \lambda^1(A)$ is continuous. Let $\lambda > 0$ be arbitrary. Since $S$ is continuous, there is $p \in  \operatorname{csn}(E)$ such that
\begin{align*}
&\| (c_i)_{i \in \Z^d} \|_{l^1(a^\lambda)} = \sum_{i \in \Z^d} |c_i| \|e_i\|_{l^1(a^\lambda)} = \sum_{i \in \Z^d} |\langle \eta_i, T((c_i)_{i \in \Z^d}) \rangle| \|S(T(e_i))\|_{l^\infty(a^\lambda)} \\
&\leq  \sum_{i \in \Z^d} |\langle \eta_i, T((c_i)_{i \in \Z^d}) \rangle| p(T(e_i))
\end{align*}
for all $(c_i)_{i \in \Z^d} \in \lambda^1(A)$, whence the continuity of $T^{-1}$ follows from \eqref{absolute}. 

$(b)$ By transposing, we obtain continuous linear mappings $T^t: E'_\beta \rightarrow (\lambda^1\{A\})'_\beta$ and  $S:  (\lambda^\infty\{A\})'_\beta \rightarrow E'_\beta$ such that $T^t \circ S^t = \iota^t$.
Consider the K\"othe set $A^\circ = \{ 1/a^{1/\lambda} \, | \, \lambda > 0\}$ and the  natural continuous embeddings $\iota_1: \lambda^1(A^\circ) \rightarrow (\lambda^\infty\{A\})'_\beta$ 
and $\iota_2: (\lambda^1\{A\})'_\beta \rightarrow  \lambda^\infty(A^\circ)$. Note that $(\iota_2 \circ T^t) \circ (S^t \circ \iota_1) = \tau$, where $\tau: \lambda^1(A^\circ) \rightarrow \lambda^\infty(A^\circ)$ denotes the natural embedding. Hence, part $(a)$ yields that $\lambda^1(A^\circ)$ is nuclear, which is equivalent to the nuclearity of $\lambda^1\{A\}$ by Proposition \ref{nuclear-echelon}.
\end{proof}
\begin{proof}[Proof of Proposition \ref{ultra-necc-2}] By Proposition \ref{nuclear-echelon}, it suffices to show that $\lambda^1[A_{\mathscr{W}_\mathfrak{M}}]$ is nuclear. To this end, we use Lemma \ref{P-trick}  with $A = A_{\mathscr{W}_\mathfrak{M}}$ and  $E = \mathcal{S}^{[\mathfrak{M}]}_{[\mathscr{W}],q}$ (in the Roumieu case, $E'_\beta$ is nuclear as the strong dual of a nuclear $(DF)$-space).  For $r = 1$ or $r = \infty$ we define  $\mathcal{E}^{[\mathfrak{M}]}_{\operatorname{per},r}$ as the space consisting of all $\Z^d$-periodic functions $\varphi \in C^\infty(\R^d)$ such that 
$$
\sup_{\alpha \in \N^d} \frac{1}{M^\lambda_{\alpha}} \|\varphi^{(\alpha)}\|_{L^r\left([-\frac{1}{2}, \frac{1}{2}]^d \right)} < \infty
$$
 for all $\lambda > 0$ (for some $\lambda >0$). We endow $\mathcal{E}^{[\mathfrak{M}]}_{\operatorname{per},r}$  with its natural Fr\'echet space topology ($(LB)$-space topology). The mappings
$$
T_0 : \lambda^1[A_{\mathscr{W}_\mathfrak{M}}] \rightarrow \mathcal{E}^{[\mathfrak{M}]}_{\operatorname{per},\infty}, \quad T_0( (c_j)_{j \in \Z^d}) = \left [\, \xi \to \sum_{j \in \Z^d} c_j e^{-2\pi i j \cdot \xi} \, \right]
$$
and
$$
S_0 : \mathcal{E}^{[\mathfrak{M}]}_{\operatorname{per},1} \rightarrow \lambda^\infty[A_{\mathscr{W}_\mathfrak{M}}], \qquad S_0(\varphi) = \left( \int_{[-\frac{1}{2},\frac{1}{2}]^d} \varphi(\xi) e^{2\pi i j \cdot \xi} d\xi \right)_{j \in \Z^d}
$$
are continuous. Next, choose $\psi$ as in Lemma \ref{construction}$(b)$ and consider the continuous linear mapping
$$
T_1 :  \mathcal{E}^{[\mathfrak{M}]}_{\operatorname{per},\infty} \rightarrow  \mathcal{S}^{[\mathfrak{M}]}_{[\mathscr{W}],q}, \quad T_1(\varphi) = \psi \varphi.
$$
Note that $\mathscr{W}$ satisfies $[\operatorname{N}]$ by Proposition \ref{ultra-necc-1}. Hence, Lemma \ref{reverse-inclusion} yields that the mapping
$$
S_1 : \mathcal{S}^{[\mathfrak{M}]}_{[\mathscr{W}],q} \rightarrow  \mathcal{E}^{[\mathfrak{M}]}_{\operatorname{per},1}, \qquad S_1(\varphi) = \sum_{j \in \Z^{d}} \varphi( \,\cdot -j)
$$
is continuous. Finally, we define the continuous linear mappings $T = T_1 \circ T_0:  \lambda^1[A_{\mathscr{W}_\mathfrak{M}}]  \rightarrow  \mathcal{S}^{[\mathfrak{M}]}_{[\mathscr{W}],q}$ and $S = S_0 \circ S_1 : \mathcal{S}^{[\mathfrak{M}]}_{[\mathscr{W}],q} \rightarrow \lambda^\infty[A_{\mathscr{W}_\mathfrak{M}}]$. The choice of $\psi$ implies that $S \circ T = \iota$.
\end{proof}

We obtain the following two important results.
\begin{theorem}\label{ultra-char}
Let $\mathfrak{M}$ be an isotropically decomposable accelerating weight sequence system satisfying $[\operatorname{L}]$ and let $\mathscr{W}$ be a weight function system satisfying $[\operatorname{M}]$. Suppose that $\mathcal{S}^{[\mathfrak{M}]}_{[\mathscr{W}],q} \neq \{0\}$ for some $q \in [1,\infty]$. Then, the following statements are equivalent:
\begin{itemize}
\item[$(i)$] $\mathfrak{M}$ satisfies $[\mathfrak{M}.2]'$ and $\mathscr{W}$ satisfies $[\operatorname{N}]$.
\item[$(ii)$] $\mathcal{S}^{[\mathfrak{M}]}_{[\mathscr{W}],q}$ is nuclear for all $q \in [1,\infty]$.
\item[$(iii)$] $\mathcal{S}^{[\mathfrak{M}]}_{[\mathscr{W}],q}$ is nuclear for some $q \in [1,\infty]$.
\end{itemize}
\end{theorem}
\begin{proof}
$(i) \Rightarrow (ii)$ This has been shown in Theorem \ref{ultra-suff}.

$(ii) \Rightarrow (iii)$ Trivial.

$(iii) \Rightarrow (i)$ In view of Lemma \ref{non-trivial}, $\mathscr{W}$ satisfies $[\operatorname{N}]$ by  Proposition \ref{ultra-necc-1}, while $\mathfrak{M}$ satisfies $[\mathfrak{M}.2]'$ by Proposition \ref{ultra-necc-2} and Lemma \ref{M2wss}.
\end{proof}

\begin{comment}
\begin{theorem}\label{ultra-char}
Let $\mathfrak{M}$ be an isotropically decomposable weight sequence system satisfying $[\operatorname{L}]$ and let $\mathscr{W}$ be a weight function system satisfying $[\operatorname{wM}]$. Suppose that $\mathcal{S}^{[\mathfrak{M}]}_{[\mathscr{W}],q} \neq \{0\}$ for some $q \in [1,\infty]$. Consider the following statements:
\begin{itemize}
\item[$(i)$] $\mathfrak{M}$ satisfies $[\mathfrak{M}.2]'$ and $\mathscr{W}$ satisfies $[\operatorname{N}]$.
\item[$(ii)$] $\mathcal{S}^{[\mathfrak{M}]}_{[\mathscr{W}],q}$ is nuclear for all $q \in [1,\infty]$.
\item[$(iii)$] $\mathcal{S}^{[\mathfrak{M}]}_{[\mathscr{W}],q}$ is nuclear for some $q \in [1,\infty]$.
\end{itemize}
Then, $(i) \Rightarrow (ii) \Rightarrow (iii)$. If  in addition $\mathfrak{M}$ is isotropically decomponsable and accelerating, and $\mathscr{W}$ satisfies $[\operatorname{M}]$, then also $(iii) \Rightarrow (i)$.
\end{theorem}
\begin{proof}
$(i) \Rightarrow (ii)$ This has been shown in Theorem \ref{ultra-suff}.

$(ii) \Rightarrow (iii)$ Trivial.

$(iii) \Rightarrow (i)$ \emph{(if $\mathfrak{M}$ is isotropically decomponsable and accelerating and $\mathscr{W}$ satisfies $[\operatorname{M}]$)}. In view of Lemma \ref{non-trivial}, $\mathscr{W}$ satisfies $[\operatorname{N}]$ by  Proposition \ref{ultra-necc-1}, while $\mathfrak{M}$ satisfies $[\mathfrak{M}.2]'$ by Proposition \ref{ultra-necc-2} and Lemma \ref{M2wss}.
\end{proof}
\end{comment}

\begin{theorem}\label{ultra-char-2}
Let $\mathfrak{M}$ be a weight sequence system satisfying $[\operatorname{L}]$ and $[\mathfrak{M}.2]'$. Let $\mathscr{W}$ be a weight function system satisfying $[\operatorname{M}]$. Suppose that $\mathcal{S}^{[\mathfrak{M}]}_{[\mathscr{W}],q} \neq \{0\}$ for some $q \in [1,\infty]$. Then, the following statements are equivalent:
\begin{itemize}
\item[$(i)$] $\mathscr{W}$ satisfies $[\operatorname{N}]$.
\item[$(ii)$] $\mathcal{S}^{[\mathfrak{M}]}_{[\mathscr{W}],q}$ is nuclear for all $q \in [1,\infty]$.
\item[$(iii)$] $\mathcal{S}^{[\mathfrak{M}]}_{[\mathscr{W}],q}$ is nuclear for some $q \in [1,\infty]$.
\item[$(iv)$] $\mathcal{S}^{[\mathfrak{M}]}_{[\mathscr{W}],q} = \mathcal{S}^{[\mathfrak{M}]}_{[\mathscr{W}],r}$ as locally convex spaces for all $q,r \in  [1,\infty]$.
\item[$(v)$] $\mathcal{S}^{[\mathfrak{M}]}_{[\mathscr{W}],q} = \mathcal{S}^{[\mathfrak{M}]}_{[\mathscr{W}],r}$ as sets for some  $q,r \in  [1,\infty]$ with $q \neq r$.
\end{itemize}
\end{theorem}
\begin{proof}
 In view of Lemma \ref{non-trivial}, this follows from Theorem \ref{ft-char-N}, Theorem \ref{ultra-suff} and Proposition \ref{ultra-necc-1}.
\end{proof}

Note that, by Lemma \ref{lemma-1}, Theorem \ref{intro-1} and Theorem \ref{intro-2} stated in the introduction are immediate corollaries of Theorem \ref{ultra-char} and Theorem \ref{ultra-char-2}, respectively. Actually, the isotropy of the weight sequences can be relaxed there to $M$ and $A$ being isotropically decomposable. Furthermore, combining Theorem \ref{ultra-char-2} with Lemma \ref{lemma-2}, we obtain the following result for spaces of Beurling-Bj\"{o}rck type.

\begin{theorem}\label{intro-3}
Let $\omega$ be a BMT weight function and let $\eta$ be a non-negative non-decreasing continuous function on  $[0,\infty)$ satisfying $(\alpha)$. Suppose that  $\mathcal{S}^{[\omega]}_{[\eta],q} \neq \{0\}$ for some $q \in [1,\infty]$. Then, the following statements are equivalent:
\begin{itemize}
\item[$(i)$] $\eta$ satisfies $[\gamma]$.
\item[$(ii)$] $\mathcal{S}^{[\omega]}_{[\eta],q}$ is nuclear for all $q \in [1,\infty]$.
\item[$(iii)$] $\mathcal{S}^{[\omega]}_{[\eta],q}$ is nuclear for some $q \in [1,\infty]$.
\item[$(iv)$] $\mathcal{S}^{[\omega]}_{[\eta],q}$ =  $\mathcal{S}^{[\omega]}_{[\eta],r}$ as locally convex spaces for all $q,r \in  [1,\infty]$.
\item[$(v)$] $\mathcal{S}^{[\omega]}_{[\eta],q}$ =  $\mathcal{S}^{[\omega]}_{[\eta],r}$ as sets for some  $q,r \in  [1,\infty]$ with $q \neq r$.
\end{itemize}
\end{theorem} 

\section{Projective description}
In this auxiliary section, we provide a projective description of the Gelfand-Shilov spaces $\mathcal{S}^{\{\mathfrak{M}\}}_{\{\mathscr{W}\}}$. This result will be used in the next section to prove kernel theorems. 

We start by recalling some basic results about the projective description of weighted $(LB)$-spaces of continuous functions \cite{Bierstedt}.  Let $X$ be a completely regular Hausdorff space. Given a non-negative function $v$ on $X$, we write $Cv(X)$ for the seminormed space consisting of all $f \in C(X)$ such that $\|f\|_v = \sup_{x \in X} |f(x)|v(x) < \infty$. If $v$ is positive and continuous, then $Cv(X)$ is a Banach space.  A family $\mathscr{V} = \{ v^{\lambda} \,|\, \lambda \in \R_{+} \}$ consisting of positive continuous functions $v^{\lambda}$ on $X$ such that $v^{\lambda}(x) \leq v^{\mu}(x)$ for all $x \in X$ and $\mu \leq \lambda$ is said to be a \emph{Nachbin family on $X$}. We define the associated $(LB)$-space
$$
\mathscr{V}C(X)= \varinjlim_{\lambda \to \infty} Cv^\lambda(X).
$$
The \emph{maximal Nachbin family associated with $\mathscr{V}$}, denoted by $ \overline{V}=\overline{V}(\mathscr{V}) $, is given by the space consisting of all non-negative upper semicontinuous functions $v$ on $X$ such that $\sup_{x \in X} v(x)/v^\lambda(x) < \infty$ for all $\lambda \in \R_+$. The \emph{projective hull of $\mathscr{V}C(X)$} is defined as the space $C\overline{V}(X)$ consisting of all $f \in C(X)$ such that $\|f\|_v < \infty$ for all 
$v \in \overline{V}$. We endow $C\overline{V}(X)$ with the locally convex topology generated by the system of seminorms $\{ \| \, \cdot \, \|_v \, | \, v \in \overline{V} \}$. The spaces $\mathscr{V}C(X)$ and $C\overline{V}(X)$ are always equal as sets. If $\mathscr{V}$ satisfies  the condition
\cite[p.\ 94]{Bierstedt}
\begin{itemize}
\item[$(\operatorname{S})$] $\forall \lambda \in \R_+  \, \exists \mu \in \R_+ \, : \, v^\mu/v^\lambda \mbox{ vanishes at infinity}$,
\end{itemize}
then these spaces also coincide topologically \cite[Corollary 5, p.\ 116]{Bierstedt}.

Let $X_j$ be a completely regular Hausdorff space and let $\mathscr{V}_j = \{ v^{\lambda}_j \,|\, \lambda \in \R_{+} \}$ be a Nachbin family on $X_j$ for $j = 1,2$. Similarly as in Section \ref{sect-2}, we denote by $\mathscr{V}_1 \otimes \mathscr{V}_2$ the Nachbin family  $\{v^\lambda_1 \otimes v^\lambda_2 \, | \, \lambda \in \R_+ \}$ on $X_1 \times X_2$, where $v^\lambda_1 \otimes v^\lambda_2(x_1,x_2) = v^\lambda_1(x_1)v^\lambda_2(x_2)$, $x_1 \in X_1, x_2 \in X_2$. Note that $\mathscr{V}_1 \otimes \mathscr{V}_2$ satisfies $(\operatorname{S})$ if and only if both $\mathscr{V}_1$ \and $\mathscr{V}_2$ do so. Moreover, $\overline{V}(\mathscr{V}_1) \otimes \overline{V}(\mathscr{V}_2)$ is upwards dense in $\overline{V}(\mathscr{V}_1 \otimes \mathscr{V}_2)$, that is, for every $v \in \overline{V}(\mathscr{V}_1 \otimes \mathscr{V}_2)$ there are $v_j \in \overline{V}(\mathscr{V}_j)$, $j = 1,2$, such that $v(x_1,x_2) \leq v_1 \otimes v_2(x_1,x_2)$ for all $x_1 \in X_1, x_2 \in X_2$. 

Note that every weight function system $\mathscr{W}$ on $\R^d$ is a Nachbin family on $X = \R^d$. Lemma \ref{lemma-C0} implies that  $\mathscr{W}$ satisfies $(\operatorname{S})$ if   $\{\operatorname{N}\}$ and $\{\operatorname{wM}\}$ hold for $\mathscr{W}$. Likewise,  a weight sequence system $\mathfrak{M}$ on $\N^d$ defines a Nachbin family on $X = \N^d$ via
$
\M^\circ = \{ 1/M^\lambda \, | \, \lambda \in \R^+\}.
$
If  $\mathfrak{M}$ satisfies $\{L\}$, then $\mathfrak{M}^\circ$ satisfies $(\operatorname{S})$. We define $\overline{V}(\mathfrak{M})$ as the family consisting of all  sequences $M$ of positive numbers such that $1/M \in \overline{V}(\mathfrak{M}^\circ)$. More concretely,  $\overline{V}(\mathfrak{M})$ consists of all sequences $M$ of positive numbers such that $\sup_{\alpha \in \N^d} M^\lambda_\alpha/M_\alpha < \infty$ for all $\lambda\in\mathbb{R}_{+}$. Furthermore, a sequence $(a_\alpha)_{\alpha \in \N^d}$ of positive numbers satisfies $\sup_{\alpha \in \N^d} a_\alpha/M^\lambda_\alpha < \infty$ for some $\lambda \in \R_+$ if and only if $\sup_{\alpha \in \N^d} a_\alpha/M_\alpha < \infty$ for all $M \in \overline{V}(\mathfrak{M})$.

\begin{remark} Denote by $\mathfrak{R}$  the family consisting of all non-decreasing sequences $(r_j)_{j \in \N}$ of positive numbers such that $r_j \to \infty$ as $j \to \infty$. Let $M$ be a weight sequence. Then, the set 
\begin{equation}
\{ (M_\alpha\prod_{j=0}^{|\alpha|} r_j)_{\alpha \in \N^d} \, | \, (r_j)_{j \in \N} \in \mathfrak{R}\}
\label{proj-r}
\end{equation}
 is downwards dense in $\overline{V}(\mathfrak{M}_M)$ (cf.\ \cite[Lemma 3.4]{Komatsu3}). The family $\mathfrak{R}$ was introduced by Komatsu to obtain a projective description of the space $\mathcal{E}^{\{M\}}(\Omega)$ of ultradifferentiable functions of Roumieu type \cite[Proposition 3.5]{Komatsu3}. Later on, this family was also used by Pilipovi\'c to give a projective description of the Gelfand-Shilov spaces $\mathcal{S}^{\{M\}}_{\{A\}}$ \cite[Lemma 4]{Pil94}.  For general weight sequence systems $\mathfrak{M}$, the family $\overline{V}(\mathfrak{M})$ is the natural generalization of the family in \eqref{proj-r}.
\end{remark}

We are ready to state and prove the main result of this section.
\begin{theorem}\label{projective-description}
Let $\mathfrak{M}$ be a weight sequence system satisfying $\{\operatorname{L}\}$ and $\{\mathfrak{M}.2\}'$, and let $\mathscr{W}$ be a weight function system satisfying $\{\operatorname{wM}\}$ and $\{\operatorname{N}\}$. Then, $\varphi \in C^\infty(\R^d)$ belongs to $\mathcal{S}^{\{\mathfrak{M}\}}_{\{\mathscr{W}\}}$ if and only if  $\| \varphi \|_{\mathcal{S}^{M}_{w,\infty}} < \infty$ for all $M \in \overline{V}(\mathfrak{M})$ and $w \in \overline{V}(\mathscr{W})$. Moreover, the topology of  $\mathcal{S}^{\{\mathfrak{M}\}}_{\{\mathscr{W}\}}$ is generated by the system of seminorms $\{ \| \, \cdot \, \|_{\mathcal{S}^{M}_{w,\infty}} \, | \,  M \in \overline{V}(\mathfrak{M}), w \in \overline{V}(\mathscr{W})\}$.
\end{theorem}

We define
$$
\iota: C^\infty(\R^d) \rightarrow C(\N^d \times \R^d), \qquad \iota(\varphi) = [(\alpha, x) \mapsto \varphi^{(\alpha)}(x)].
$$
The proof of Theorem \ref{projective-description} is based on the ensuing lemma.
\begin{lemma}\label{embedding}
Let $\mathfrak{M}$ be a weight sequence system satisfying $\{\operatorname{L}\}$ and $\{\mathfrak{M}.2\}'$, and let $\mathscr{W}$ be a weight function system satisfying $\{\operatorname{wM}\}$ and $\{\operatorname{N}\}$. Then, $\varphi \in C^\infty(\R^d)$ belongs to  $\mathcal{S}^{\{\mathfrak{M}\}}_{\{\mathscr{W}\}}$ if and only if $\iota(\varphi) \in  (\M^\circ \otimes \mathscr{W}) C(\N^d \times \R^d)$. Moreover, 
$$
\iota: \mathcal{S}^{\{\mathfrak{M}\}}_{\{\mathscr{W}\}} \rightarrow (\M^\circ \otimes \mathscr{W}) C(\N^d \times \R^d)
$$
is a topological embedding.
\end{lemma}
\begin{proof}
The first part and the fact that $\iota$ is continuous are obvious. We now show that $\iota$ is a topological embedding. Fix an arbirtrary $q \in (1,\infty)$. For $n \in \Z_+$ we write $X_n$ for the Banach space consisting of all $\varphi \in C^\infty(\R^d)$ such that 
$$
\| \varphi \|_{X_n} = \left( \sum_{\alpha \in \N^d} \left( \frac{\| \varphi^{(\alpha)} w^n\|_{L^q} }{M^n_\alpha} \right)^q \right)^{1/q} < \infty
$$
and $Y_n$ for the Banach space consisting of all sequences $(\varphi_{\alpha})_{\alpha \in \N^d}$ of measurable functions such that
$$
\| (\varphi_{\alpha})_{\alpha \in \N^d} \|_{Y_n} = \left( \sum_{\alpha \in \N^d} \left( \frac{\| \varphi_{\alpha} w^n\|_{L^q} }{M^n_\alpha} \right)^q \right)^{1/q} < \infty.
$$
Note that both $X_n$ and $Y_n$ are reflexive. The mapping $\rho_n: X_n \rightarrow Y_n, \varphi \mapsto (\varphi^{(\alpha)})_{\alpha \in \N^d}$ is a topological embedding. Set $X = \varinjlim_{n \in \Z_+} X_n$ and $Y = \varinjlim_{n \in \Z_+} Y_n$.  Condition $\{\operatorname{L}\}$ implies that $X =  \mathcal{S}^{\{\mathfrak{M}\}}_{\{\mathscr{W}\},q} = \mathcal{S}^{\{\mathfrak{M}\}}_{\{\mathscr{W}\}}$ as locally convex spaces. We claim that $\rho = \varinjlim_{n \in \Z_+} \rho_n : X = \mathcal{S}^{\{\mathfrak{M}\}}_{\{\mathscr{W}\}}   \rightarrow Y$ is a topological embedding. Before we prove the claim, let us show how it entails the result. Condition $\{L\}$ and Lemma \ref{lemma-C0} imply that the mapping
$$
\tau: (\M^\circ \otimes \mathscr{W}) C(\N^d \times \R^d) \rightarrow  Y, f \mapsto (f(\alpha, \, \cdot \,))_{\alpha \in \N^d}
$$
is well-defined and continuous. Note that $\rho = \tau \circ \iota$. Hence, $\iota$ is a topological embedding because $\rho$ is so. We now show the claim with the aid of the dual Mittag-Leffler theorem \cite[Lemma 1.4]{Komatsu}. For $n \in \Z_+$ we set $Z_n = Y_n / \rho_n(X_n)$. Hence, $Z_n$ is a reflexive Banach space.  We denote by $\pi_n : Y_n \rightarrow Z_n$ the quotient mapping. The natural linking mappings $Z_n \rightarrow Z_{n+1}$ are injective since
$\rho_{n+1}(X_{n+1}) \cap Y_n = \rho_{n}(X_n)$. Consider the following injective inductive sequence of short topologically exact sequences
\begin{center}
\begin{tikzpicture}
  \matrix (m) [matrix of math nodes, row sep=2em, column sep=2em]
    {
    0 & X_1 & Y_1  & Z_1 & 0 \\
   0 & X_2 & Y_2  & Z_2 & 0  \\ 
        & \vdots & \vdots & \vdots &  \\ };
  
  { [start chain] \chainin (m-1-1);
\chainin (m-1-2);
\chainin (m-1-3)[join={node[above,labeled] {\rho_1}}];
\chainin (m-1-4)[join={node[above,labeled] {\pi_1}}];
\chainin (m-1-5); }
  
  { [start chain] \chainin (m-2-1);
\chainin (m-2-2);
\chainin (m-2-3)[join={node[above,labeled] {\rho_2}}];
\chainin (m-2-4)[join={node[above,labeled] {\pi_2}}];
\chainin (m-2-5); }

{ [start chain] \chainin (m-1-2);
\chainin (m-2-2);
\chainin (m-3-2);
 }

{ [start chain] \chainin (m-1-3);
\chainin (m-2-3);
\chainin (m-3-3);
 }

{ [start chain] \chainin (m-1-4);
\chainin (m-2-4);
\chainin (m-3-4);
 }
\end{tikzpicture}
\end{center}
The linking mappings of the inductive spectra $(X_n)_{n \in \Z_+}$, $(Y_n)_{n \in \Z_+}$ and $(Z_n)_{n \in \Z_+}$ are weakly compact as continuous linear mappings between reflexive Banach spaces. In particular, these inductive spectra are regular. Furthermore, $\varinjlim_{n \in \Z_+} X_n = X = \mathcal{S}^{\{\mathfrak{M}\}}_{\{\mathscr{W}\}}$ is Montel since it is a nuclear $(DF)$-space. Therefore, the dual Mittag-Lefller theorem \cite[Lemma 1.4]{Komatsu} yields that $\rho = \varinjlim_{n \in \Z_+} \rho_n$ is a topological embedding.
\end{proof}

\begin{proof}[Proof of Theorem \ref{projective-description}] We write $E$ for the space consisting of all $\varphi \in C^\infty(\R^d)$ such that $\| \varphi \|_{\mathcal{S}^{M}_{w,\infty}} < \infty$ for all $M \in \overline{V}(\mathfrak{M})$ and $w \in \overline{V}(\mathscr{W})$ endowed with the locally convex topology generated by the  system of seminorms $\{ \| \, \cdot \, \|_{\mathcal{S}^{M}_{w,\infty}} \, | \,  M \in \overline{V}(\mathfrak{M}), w \in \overline{V}(\mathscr{W})\}$. We need to show that 
$\mathcal{S}^{\{\mathfrak{M}\}}_{\{\mathscr{W}\}}$ and $E$ coincide as locally convex spaces. Since $\overline{V}(\M^\circ) \otimes \overline{V}(\mathscr{W})$ is upward dense in $\overline{V}(\M^\circ \otimes \mathscr{W})$, we have that $\varphi \in C^\infty(\R^d)$ belongs to  $E$ if and only if $\iota(\varphi) \in C\overline{V}(\M^\circ \otimes \mathscr{W})(\R^d)$ and that  
$\iota: E \rightarrow C\overline{V}(\M^\circ \otimes \mathscr{W})(\R^d)$ is a topological embedding. As both the Nachbin families $\M^\circ$ and $\W$ satisfy $(\operatorname{S})$,  $\M^\circ \otimes \mathscr{W}$ does so as well. Hence, $(\M^\circ \otimes \mathscr{W}) C(\N^d \times \R^d) = C\overline{V}(\M^\circ \otimes \mathscr{W})(\R^d)$ as locally convex spaces. The result now follows from Lemma \ref{embedding}.
\end{proof}

\section{Kernel theorems}\label{section kernel}
We prove  kernel theorems for the spaces $\mathcal{S}^{[\mathfrak{M}]}_{[\mathscr{W}]}$  in this section. To this end, we introduce vector-valued versions of   $\mathcal{S}^{[\mathfrak{M}]}_{[\mathscr{W}]}$ and give a tensor product representation for them. 

We start by briefly recalling some notions about the $\varepsilon$-product and tensor products \cite{Schwartz-v, Komatsu3}.  Given two lcHs  $E$ and $F$, we denote by $\mathcal{L}(E, F)$ the space consisting of all continuous linear mappings from $E$ into $F$. 
The $\varepsilon$-product of $E$ and $F$ is defined as $E \varepsilon F = \mathcal{L}(F'_c, E)$ endowed with the topology of uniform convergence over the equicontinuous subsets of $F'$,  where the subscript $c$ indicates that we endow $F'$ with the topology of uniform convergence on balanced convex compact subsets of $F$. The spaces $E \varepsilon F$ and $F \varepsilon E$ are canonically isomorphic as locally convex spaces \cite[p.\ 657]{Komatsu3}.
   If $F$ is Montel, then $E \varepsilon F = \mathcal{L}_\beta(F'_\beta, E)$. We write 
   $E \otimes_\varepsilon F$ and $E \otimes_\pi F$ to indicate that we endow the  tensor product $E \otimes F$ with the $\varepsilon$-topology and the projective topology, respectively. If either $E$ or $F$ is nuclear, we have that $E \otimes_\varepsilon F = E \otimes_\pi F$ and we drop the subscripts $\varepsilon$ and $\pi$ in the  notation. The tensor product $E \otimes F$ is canonically embedded into $E\varepsilon F$ and the induced topology on $E \otimes F$ is the $\varepsilon$-topology.  If $E$ and $F$ are complete and if either $E$ or $F$ has the weak approximation property (in particular, if either $E$ or $F$ is nuclear), then $E \varepsilon F$ and  $E \widehat{\otimes}_\varepsilon F$ are canonically isomorphic as locally convex spaces  \cite[Prop.\ 1.4]{Komatsu3}.

Let us now introduce vector-valued Gelfand-Shilov spaces. Let $\mathfrak{M}$ be a weight sequence system, let $\mathscr{W}$ be a weight function system and let $E$ be a lcHs. We define $\mathcal{S}^{[\mathfrak{M}]}_{[\mathscr{W}], \infty}(\R^d;E) = \mathcal{S}^{[\mathfrak{M}]}_{[\mathscr{W}]}(\R^d;E)$ as the space consisting of all ${\bm{\varphi}} \in C^\infty(\R^d;E)$ such that for all $p \in \csn(E)$ and $\lambda \in \R_+$ (for all $p \in \csn(E)$, $M \in \overline{V}(\M)$ and $w \in \overline{V}(\W)$)
\begin{gather*}
p_\lambda({\bm{\varphi}}) = \sup_{\alpha \in \N^d} \sup_{x \in \R^d} \frac{p({\bm{\varphi}}^{(\alpha)}(x))w^\lambda(x)}{M^\lambda_{\alpha}} < \infty \\
\left ( p_{M,w}({\bm{\varphi}}) = \sup_{\alpha \in \N^d} \sup_{x \in \R^d} \frac{p({\bm{\varphi}}^{(\alpha)}(x))w(x)}{M_{\alpha}}
 < \infty \right).
 \end{gather*}
 We endow $\mathcal{S}^{[\mathfrak{M}]}_{[\mathscr{W}]}(\R^d;E)$ with the locally convex topology generated by the system of seminorms $\{ p_\lambda \, | \, p \in \csn(E), \lambda \in \R_+ \}$ ($\{ p_{M,w} \, | \, p \in \csn(E),  M \in \overline{V}(\M), w \in \overline{V}(\W) \}$).
 
\begin{proposition}\label{charvv}
Let $\mathfrak{M}$ be a weight sequence system satisfying $[\operatorname{L}]$ and $[\mathfrak{M}.2]'$, let $\mathscr{W}$ be a weight function system satisfying $[\operatorname{wM}]$ and $[\operatorname{N}]$, and let $E$ be a complete lcHs. Then, the following canonical isomorphisms of locally convex spaces hold
$$
\mathcal{S}^{[\mathfrak{M}]}_{[\mathscr{W}]}(\R^d;E)  \cong \mathcal{S}^{[\mathfrak{M}]}_{[\mathscr{W}]}(\R^d) \varepsilon E \cong \mathcal{S}^{[\mathfrak{M}]}_{[\mathscr{W}]}(\R^d) \widehat{\otimes}  E. 
$$
\end{proposition}
We will make use of the ensuing result of Komatsu \cite{Komatsu3} to show Proposition \ref{charvv}.
\begin{lemma}[{\cite[Lemma 1.12]{Komatsu3}}] \label{criteriumKomatsu} Let $G$ be a semi-Montel lcHs such that $G$ is continuously included in $C(\R^d)$ and let $E$ be a complete lcHs. Then, every function  ${\bm{\varphi}} \in C(\R^d; E)$ satisfying 
\begin{equation}
\langle e', {\bm{\varphi}} \rangle : \R^d \rightarrow \C, \ x \mapsto \langle e', {\bm{\varphi}}(x) \rangle \mbox{ belongs to $G$ for all $e' \in E'$}
\label{property-K}
\end{equation}
defines an element of  $G \varepsilon E$ via $E'_c \rightarrow G,$ $e' \mapsto  \langle e', {\bm{\varphi}} \rangle$. Conversely, for every $T \in  G \varepsilon E$ there is a unique ${\bm{\varphi}} \in  C(\R^d; E)$ satisfying \eqref{property-K} such that $T(e') = \langle e', {\bm{\varphi}} \rangle$ for all $e' \in E'$.
\end{lemma}

\begin{proof}[Proof of Proposition \ref{charvv}] We only show the Roumieu case as the Beurling case is similar. The second isomorphism follows from the fact that $\mathcal{S}^{\{\mathfrak{M}\}}_{\{\mathscr{W}\}}(\R^d)$ is complete and nuclear (recall that every nuclear $(DF)$-space is complete). We now show the first isomorphism. This amounts to showing that the mapping
\begin{equation}
\mathcal{S}^{\{\mathfrak{M}\}}_{\{\mathscr{W}\}}(\R^d;E) \rightarrow \mathcal{S}^{\{\mathfrak{M}\}}_{\{\mathscr{W}\}}(\R^d) \varepsilon E, \ \bm{\varphi} \mapsto [e' \mapsto \langle e', {\bm{\varphi}} \rangle]
\label{iso-1}
\end{equation}
is a topological isomorphism. We first show that it is a well-defined bijective mapping.
By Lemma \ref{criteriumKomatsu} with $G = \mathcal{S}^{\{\mathfrak{M}\}}_{\{\mathscr{W}\}}(\R^d)$ ($G$ is semi-Montel because it is nuclear), it suffices to show that a function ${\bm{\varphi}} \in  C(\R^d; E)$ belongs to $\mathcal{S}^{\{\mathfrak{M}\}}_{\{\mathscr{W}\}}(\R^d;E)$ if and only if $\langle e', {\bm{\varphi}} \rangle \in \mathcal{S}^{\{\mathfrak{M}\}}_{\{\mathscr{W}\}}(\R^d)$ for all $e' \in E'$. The direct implication is obvious. Conversely, let ${\bm{\varphi}} \in C(\R^d; E)$ be such that  $\langle e', {\bm{\varphi}} \rangle \in \mathcal{S}^{\{\mathfrak{M}\}}_{\{\mathscr{W}\}}(\R^d)$ for all $e' \in E'$. In particular,  $\langle e', {\bm{\varphi}} \rangle \in C^\infty(\R^d)$ for all $e' \in E'$. By \cite[Appendice Lemme II]{Schwartz-v}, we have that $\bm{\varphi} \in C^\infty(\R^d;E)$ and 
$$
\langle e', {\bm{\varphi}} \rangle ^{(\alpha)} = \langle e', {\bm{\varphi}}^{(\alpha)} \rangle, \qquad e' \in E', \alpha \in \N^d.
$$
Theorem \ref{projective-description} implies that for all $M \in \overline{V}(\mathfrak{M})$ and $w \in \overline{V}(\W)$  the set
$$
\left \{  \frac{\bm{\varphi}^{(\alpha)}(x)w(x)}{M_\alpha} \, | \, x \in \R^d, \alpha \in \N^d    \right\}
$$
is weakly bounded in $E$. Hence, this set is bounded in $E$ by Mackey's theorem. This  
 means that $\bm{\varphi} \in \mathcal{S}^{\{\mathfrak{M}\}}_{\{\mathscr{W}\}}(\R^d;E)$.  Next, we show that the isomorphism in \eqref{iso-1} holds topologically. Let $M \in \overline{V}(\mathfrak{M})$, $w \in \overline{V}(\W)$ and $p \in \csn(E)$ be arbitrary.  We denote by $B$ the polar set of the $p$-unit ball in $E$. The bipolar theorem yields that
\begin{align*}
 \sup_{e' \in B} \| \langle e', \bm{\varphi}\rangle \|_{\mathcal{S}^{M}_{w,\infty}} = \sup \left \{  \frac{|\langle e',\bm{\varphi}^{(\alpha)}(x) \rangle| w(x)}{M_\alpha} \, | \,   x \in \R^d, \alpha \in \N^d, e' \in B  \right\}  
= p_{M,w}(\bm{\varphi})
\end{align*} 
for all $\bm{\varphi} \in \mathcal{S}^{\{\mathfrak{M}\}}_{\{\mathscr{W}\}}(\R^d;E)$. The result now follows from Proposition \ref{projective-description}.
\end{proof}
We are ready to prove the kernel theorems.
\begin{theorem}
Let $\M_j$ be a weight sequence system on $\N^{d_j}$ satisfying $[\operatorname{L}]$ and $[\mathfrak{M}.2]'$, and let $\mathscr{W}_j$ be a weight function system on $\R^{d_j}$  satisfying $[\operatorname{wM}]$ and $[\operatorname{N}]$ for $j = 1,2$. The following canonical isomorphisms of locally convex spaces hold
\begin{equation}
\mathcal{S}^{[\mathfrak{M}_1 \otimes \mathfrak{M}_2 ]}_{[\mathscr{W}_1 \otimes \mathscr{W}_2]}(\R^{d_1+d_2}) \cong   \mathcal{S}^{[\mathfrak{M}_1]}_{[\mathscr{W}_1]}(\R^{d_1}) \widehat{\otimes} \mathcal{S}^{[\mathfrak{M}_2]}_{[\mathscr{W}_2]}(\R^{d_2}) \cong \mathcal{L}_\beta(\mathcal{S}^{[\mathfrak{M}_1]}_{[\mathscr{W}_1]}(\R^{d_1})'_\beta,  \mathcal{S}^{[\mathfrak{M}_2]}_{[\mathscr{W}_2]}(\R^{d_2}))
\label{kernel-1}
\end{equation}
and
\begin{equation}
\mathcal{S}^{[\mathfrak{M}_1 \otimes \mathfrak{M}_2 ]}_{[\mathscr{W}_1 \otimes \mathscr{W}_2]}(\R^{d_1+d_2})'_\beta \cong   \mathcal{S}^{[\mathfrak{M}_1]}_{[\mathscr{W}_1]}(\R^{d_1})'_\beta \widehat{\otimes} \mathcal{S}^{[\mathfrak{M}_2]}_{[\mathscr{W}_2]}(\R^{d_2})'_\beta \cong \mathcal{L}_\beta(\mathcal{S}^{[\mathfrak{M}_1]}_{[\mathscr{W}_1]}(\R^{d_1}),  \mathcal{S}^{[\mathfrak{M}_2]}_{[\mathscr{W}_2]}(\R^{d_2})'_\beta).
\label{kernel-2}
\end{equation}
\end{theorem}
\begin{proof}
The isomorphisms in \eqref{kernel-2} follow from those in \eqref{kernel-1} and the general theory of nuclear Fr\'echet and $(DF)$-spaces, see e.g.\ \cite[Theorem 2.2]{Komatsu3}. We now show the isomorphisms in \eqref{kernel-1}. By Proposition  \ref{charvv} and the fact that $\mathcal{S}^{[\mathfrak{M}_1]}_{[\mathscr{W}_1]}(\R^{d_1})$ is Montel (as it is nuclear and barreled), it is enough to show that the following canonical isomorphism  of locally convex spaces holds
$$\mathcal{S}^{[\mathfrak{M}_1 \otimes \mathfrak{M}_2 ]}_{[\mathscr{W}_1 \otimes \mathscr{W}_2]}(\R^{d_1+d_2}) \cong \mathcal{S}^{[\mathfrak{M}_1]}_{[\mathscr{W}_1]}(\R^{d_1};\mathcal{S}^{[\mathfrak{M}_2]}_{[\mathscr{W}_2]}(\R^{d_2})).$$
This amounts to verify that the mappings
$$
\mathcal{S}^{[\mathfrak{M}_1 \otimes \mathfrak{M}_2 ]}_{[\mathscr{W}_1 \otimes \mathscr{W}_2]}(\R^{d_1+d_2}) \rightarrow \mathcal{S}^{[\mathfrak{M}_1]}_{[\mathscr{W}_1]}(\R^{d_1};\mathcal{S}^{[\mathfrak{M}_2]}_{[\mathscr{W}_2]}(\R^{d_2})): \varphi \mapsto [x_1 \mapsto \varphi(x_1, \, \cdot \,)]  
$$
and
$$
\mathcal{S}^{[\mathfrak{M}_1]}_{[\mathscr{W}_1]}(\R^{d_1};\mathcal{S}^{[\mathfrak{M}_2]}_{[\mathscr{W}_2]}(\R^{d_2})) \rightarrow \mathcal{S}^{[\mathfrak{M}_1 \otimes \mathfrak{M}_2 ]}_{[\mathscr{W}_1 \otimes \mathscr{W}_2]}(\R^{d_1+d_2}): \bm{\varphi} \mapsto [(x_1,x_2) \mapsto \bm{\varphi}(x_1)(x_2)],
$$
which are inverses of each other, are well-defined and continuous. But the proofs of these facts are standard and  therefore omitted (we only remark that in the Roumieu case one needs to use Theorem \ref{projective-description}).
\end{proof}


\begin{thebibliography}{999}
\setlength{\itemsep}{0pt}
\bibitem{Bierstedt} K.~D.~Bierstedt, \emph{An introduction to locally convex inductive limits}, pp.\ 35--133, in \emph{Functional
analysis and its applications} (Nice, 1986), World Sci. Publishing, Singapore, 1988.

\bibitem{B-S} K.~D.~Bierstedt, R.~Meise, W.~H.~Summers, \emph{K\"othe sets and K\"othe sequence spaces}, pp.\ 27--91, in \emph{Functional analysis, holomorphy and approximation theory} (Rio de Janeiro, 1980), North-Holland Math. Stud. \textbf{71}, North-Holland, Amsterdam-New York, 1982.

\bibitem{bjorck66} G.~Bj\"{o}rck, \emph{Linear partial differential operators and generalized distributions,} Ark. Mat. \textbf{6} (1966), 351--407. 


\bibitem{B-J-O2019} C.~Boiti, D.~Jornet, A.~Oliaro, \emph{The Gabor wave front set in spaces of ultradifferentiable functions,} Monatsh. Math. \textbf{188} (2019), 199--246.

\bibitem{B-J-O} C.~Boiti, D.~Jornet, A.~Oliaro, \emph{About the nuclearity of $\mathcal{S}_{(M_p)}$ and $\mathcal{S}_{\omega}$}, pp.\ 121--129, in \emph{Advances in Microlocal and Time-Frequency Analysis,} Applied and Numerical Harmonic Analysis, Birkh\"{a}user, Cham, 2020.

\bibitem{B-J-O-S}C.~Boiti, D.~Jornet, A.~Oliaro, G.~Schindl, \emph{Nuclearity of rapidly decreasing ultradifferentiable functions and time-frequency analysis,} preprint, arXiv:1906.05171.


\bibitem{B-M-T} R.~W.~Braun, R.~Meise, B.~A.~Taylor, \emph{Ultradifferentiable functions and Fourier analysis}, Results Math. \textbf{17} (1990), 206--237.

\bibitem{C-G-P-R} M.~Cappiello, T.~Gramchev, S.~Pilipovi\'{c}, L.~Rodino, \emph{Anisotropic Shubin operators and eigenfunction expansions in Gelfand-Shilov spaces,} J. Anal. Math. \textbf{138} (2019), 857--870. 

\bibitem{C-T} M.~Cappiello, J.~Toft, \emph{Pseudo-differential operators in a Gelfand-Shilov setting,} Math. Nachr. \textbf{290} (2017), 738--755. 

\bibitem{C-K-P} R.~Carmichael, A.~Kami\'nski, S.~Pilipovi\'c, \emph{Boundary values and convolution in ultradistribution
spaces}, World Scientific Publishing Co. Pte. Ltd., Hackensack, NJ, 2007.

\bibitem{D-N-V-NBB} A.~Debrouwere, L.~Neyt, J.~Vindas, \emph{Characterization of nuclearity for Beurling-Bj\"{o}rck spaces}, preprint, arXiv:1908.10886.

\bibitem{Gelfand-Shilov3} I.~M.~Gel'fand, G.~E.~Shilov, {\em Generalized functions. Vol. 3: Theory of differential equations,} Academic Press, New York-London, 1967.

\bibitem{Gelfand-Shilov2}  I.~M.~Gel'fand , G.~E.~Shilov, {\em Generalized functions. Vol. 2: Spaces of fundamental and generalized functions,} Academic Press, New York-London, 1968.

\bibitem{Gramchev2016} T.~Gramchev, \emph{Gelfand-Shilov spaces: structural properties and applications to pseudodifferential operators in $\mathbb{R}^{n}$,} pp. 1--68, in \emph{Quantization, PDEs, and geometry,} Oper. Theory Adv. Appl., 251, Adv. Partial Differ. Equ. (Basel), Birkh\"{a}user/Springer, Cham, 2016.

\bibitem{Grothendieck} A.~Grothendieck, \emph{Produits tensoriels topologiques et espaces nucl\'{e}aires}, Mem. Amer. Math. Soc. \textbf{16}, 1955.

\bibitem{Komatsu} H.~Komatsu, \emph{Ultradistributions I. Structure theorems and a characterization}, J. Fac. Sci. Univ. Tokyo Sect. IA Math. \textbf{20} (1973), 25--105. 

\bibitem{Komatsu3} H.~Komatsu, \emph{Ultradistributions. III. Vector-valued ultradistributions and the theory of kernels}, J. Fac. Sci. Univ. Tokyo Sect. IA Math. \textbf{29} (1982), 653--717.

\bibitem{Langenbruch2006} M.~Langenbruch, \emph{Hermite functions and weighted spaces of generalized functions,} Manuscripta Math. \textbf{119} (2006),  269--285.

\bibitem{M-V} R.~Meise, D.~Vogt, \emph{Introduction to functional analysis}, Clarendon Press, Oxford University Press, New York, 1997.

\bibitem{Mitjagin1960} B.~S.~Mityagin, \emph{Nuclearity and other properties of spaces of type $S$,} Trudy Moskov. Mat. Ob\u{s}\u{c}. \textbf{9} (1960), 317--328.


\bibitem{N-R} F.~Nicola, L.~Rodino, \emph{Global pseudo-differential calculus on Euclidean spaces,}  Birkh\"{a}user Verlag, Basel, 2010.

\bibitem{Petzsche} H.~J.~Petzsche, \emph{Die nuklearit\"at der ultradistributionsr\"aume und der satz vom kern I}, Manuscripta Math.
\textbf{24} (1978), 133--171.

\bibitem{Pietsch} A.~Pietsch, \emph{Nuclear locally convex spaces}, Springer-Verlag, New York-Heidelberg, 1972.

\bibitem{Pil94} S.~Pilipovi\'{c}, \emph{Characterization of bounded sets in spaces of ultradistributions}, Proc. Amer. Math. Soc. \textbf{120} (1994), 1191--1206.

\bibitem{P-P-V}S.~Pilipovi\'{c}, B.~Prangoski, J.~Vindas, \emph{On quasianalytic classes of Gelfand-Shilov type. Parametrix
and convolution}, J. Math. Pures Appl. \textbf{116} (2018), 174--210.

\bibitem{Prangoski2013} B.~Prangoski, \emph{Pseudodifferential operators of infinite order in spaces of tempered ultradistributions,} J. Pseudo-Differ. Oper. Appl. \textbf{4} (2013), 495--549. 

\bibitem{R-S-Composition} A.~Rainer, G.~Schindl, \emph{Composition in ultradifferentiable classes}, Stud. Math. \textbf{224} (2014), 97--131.

\bibitem{Schwartz-v} L.~Schwartz, \emph{Th\'eorie des distributions \`{a} valeurs vectorielles. I}, Ann. Inst. Fourier (Grenoble) \textbf{7} (1957), 1--141. 

\bibitem{Schwartz} L.~Schwartz, \emph{Th\'{e}orie des distributions}, Hermann, Paris, 1966.

\bibitem{Vogt} D.~Vogt, \emph{Sequence space representations of spaces of test functions and
distributions}, pp.\  405--443, in  \emph{Functional analysis, holomorphy, and approximation
theory}, Lecture Notes in Pure and Appl. Math. \textbf{83}, 1983.

\bibitem{V-V2016} \DJ. Vu\v{c}kovi\'{c}, J.~Vindas, \emph{Eigenfunction expansions of ultradifferentiable functions and ultradistributions in $\mathbb{R}^{n}$,} J. Pseudo-Differ. Oper. Appl. \textbf{7} (2016), 519--531.

\end{thebibliography}
\end{document}